\documentclass[a4paper]{amsart}
\usepackage[left=3cm,right=3cm,top=3cm,bottom=3cm]{geometry}

\usepackage[utf8]{inputenc}
\usepackage{amsmath}
\usepackage{amsfonts}
\usepackage{amssymb}
\usepackage{mathrsfs, multirow}
\usepackage{sseq}

\usepackage{graphicx}
\usepackage{epigraph}
\usepackage[pagebackref]{hyperref}
\usepackage{amsthm}
\usepackage{verbatim}
\usepackage{xcolor}

\usepackage{tikz}
\usepackage{tikz-cd}
\usetikzlibrary{calc}
\usetikzlibrary{decorations.pathmorphing}

\usepackage[utf8]{inputenc}
\usepackage{amsmath}
\usepackage{amsfonts}
\usepackage{amssymb}
\usepackage{mathrsfs, multirow}
\usepackage{sseq}
\usepackage{listings}

\usepackage{graphicx}
\usepackage{epigraph}
\usepackage{xypic}
\usepackage[pagebackref]{hyperref}
\usepackage{amsthm}
\usepackage{xcolor}

\usepackage{tikz}
\usepackage{tikz-cd}
\usetikzlibrary{calc}
\usetikzlibrary{decorations.pathmorphing}

\DeclareMathOperator{\supp}{supp}

\DeclareMathOperator{\Tor}{Tor}
\DeclareMathOperator{\Ext}{Ext}
\DeclareMathOperator{\Ker}{Ker}
\DeclareMathOperator{\Coker}{Coker}
\DeclareMathOperator{\Img}{Im}

\DeclareMathOperator{\sk}{sk}

\DeclareMathOperator{\colim}{colim}

\DeclareMathOperator{\lk}{lk}

\DeclareMathOperator{\chr}{char}

\newcommand{\uxa}{\ensuremath{(\underline{X},\underline{A})}} 
\newcommand{\ux}{\ensuremath{(\underline{X},\underline{\ast})}}
\newcommand{\us}{\ensuremath{(\underline{S},\underline{\ast})}}

\newcommand{\djk}{\ensuremath{DJ_{\K}}}

\DeclareMathOperator{\spn}{span}

\def\CC{\mathbb{C}}
\def\ZZ{\mathbb{Z}}
\def\FF{\mathbb{F}}
\def\QQ{\mathbb{Q}}

\def\Zg{\ZZ_{\geq 0}}
\def\Zm{\Zg^m}

\def\Z{\mathcal{Z}}
\def\K{\mathcal{ K}}
\def\L{\mathcal{L}}

\def\ZK{\Z_\K}
\def\DJ{\mathrm{DJ}}

\def\ux{(\underline{X},\underline{\ast})}
\def\uxa{(\underline{X},\underline{A})}
\def\k{\mathbf{k}}

\def\H{{\widetilde{H}}}
\def\pt{{\mathrm{pt}}}
\def\id{{\mathrm{id}}}

\def\kKc{\k\langle\K\rangle}
\def\MF{\mathrm{MF}}
\def\AMF{\mathrm{AMF}}
\def\Kf{\K^f}

\newtheorem*{theorem*}{Theorem}
\newtheorem{thm}{Theorem}[section]
\newtheorem{theorem}[thm]{Theorem}
\newtheorem{lemma}[thm]{Lemma}

\newtheorem*{thm*}{thm}
\newtheorem{lmm}[thm]{Lemma}

\newtheorem{prp}[thm]{Proposition}
\newtheorem{crl}[thm]{Corollary}

\theoremstyle{definition}

\numberwithin{equation}{section}
\newtheorem{dfn}[thm]{Definition}
\newtheorem{rmk}[thm]{Remark}
\newtheorem{exm}[thm]{Example}

\numberwithin{equation}{section}

\tikzset{curve/.style={settings={#1},to path={(\tikztostart)
    .. controls ($(\tikztostart)!\pv{pos}!(\tikztotarget)!\pv{height}!270:(\tikztotarget)$)
    and ($(\tikztostart)!1-\pv{pos}!(\tikztotarget)!\pv{height}!270:(\tikztotarget)$)
    .. (\tikztotarget)\tikztonodes}},
    settings/.code={\tikzset{quiver/.cd,#1}
        \def\pv##1{\pgfkeysvalueof{/tikz/quiver/##1}}},
    quiver/.cd,pos/.initial=0.35,height/.initial=0}

\tikzset{tail reversed/.code={\pgfsetarrowsstart{tikzcd to}}}
\tikzset{2tail/.code={\pgfsetarrowsstart{Implies[reversed]}}}
\tikzset{2tail reversed/.code={\pgfsetarrowsstart{Implies}}}
\tikzset{no body/.style={/tikz/dash pattern=on 0 off 1mm}}

\graphicspath{ {./images/} }

\hypersetup{
    colorlinks=true,
    linkcolor=blue,
    citecolor=blue,
    urlcolor=blue,
    pdftitle={A colimit decomposition for the loop homology of polyhedral products},
    pdfauthor={Lewis Stanton and Fedor Vylegzhanin},
}

\title{A Colimit decomposition for the loop homology of polyhedral products}
\author{Lewis Stanton} 
\address{Mathematical Sciences, University of Southampton, Southampton SO17 1BJ, United Kingdom}
\email{L.R.Stanton@soton.ac.uk}

\author{Fedor Vylegzhanin} 
\address{\parbox{\linewidth}{
National Research University Higher School of Economics, Moscow, Russia;\\
Steklov Mathematical Insitute of Russian Academy of Sciences, Moscow, Russia
}}
\email{vylegf@gmail.com}

\subjclass[2020]{Primary 55P35, 57S12; Secondary 16E30, 13F55, 18A30.}
\keywords{polyhedral product, loop homology, colimit, presentation, Poincar\'e series}

\begin{document}

\begin{abstract}
We show that the loop homology algebras of polyhedral products of the form $\ux^{\K}$ can be written as a colimit over the flagification of $\K$, and obtain a similar result for the Poincar\'e series. This effectively reduces the study of the algebras $H_*(\Omega\ux^{\K})$ to the case of 1-neighbourly simplicial complexes.
We give presentations of the loop homology of Davis--Januszkiewicz spaces (i.e. Yoneda algebras of Stanley--Reisner rings) and calculate the Poincar\'e series of looped polyhedral products associated to various families of simplicial complexes, including HMF-presented complexes and skeleta of flag complexes.
\end{abstract}
\maketitle

\section{Introduction}
\label{sec:intro}

Polyhedral products are a family of spaces which have been the subject of intense study recently due to their many applications across mathematics. One particular object of study has been their pointed loop space, and much work has been done to study them from various perspectives \cite{BBC20}. These include aspects of their homotopy theory \cite{Th17}, and their connections to commutative algebra \cite{vylegzhanin22} and geometric group theory \cite{PRah24}.

In this paper, we focus on commutative algebra by studying the loop homology of a family of polyhedral products. To proceed further, we state the definition of a polyhedral product. Let $\K$ be a simplicial complex on the vertex set $[m]=\{1,\dots,m\}$, and let $(X_1,A_1),\cdots,(X_m,A_m)$ be pairs of pointed $CW$-complexes. For a face $I \in \K$, define $\uxa^{I} = \prod_{i=1}^m Y_i^{I}$, where $Y_i^{I} = X_i$ if $i \in I$, and $Y_i^{I} = A_i$ if $i \notin I$. The \textit{polyhedral product} determined by $\K$ and $\uxa$ is \[\uxa^\K =\bigcup\limits_{I \in \K} \uxa^{I}\subseteq \prod\limits_{i=1}^m X_i.\] An important special case is when each $(X_i,A_i) = (\mathbb{C}P^\infty,*)$. In this case, the polyhedral product is called the \textit{Davis-Januskiewicz space}, and is denoted by $\djk$. This is closely related to the \textit{moment-angle complex}, which is the polyhedral product $\ZK := (D^2,S^1)^\K$.

In this paper, the main object of study is $\Omega \ux^\K$. When $\K$ is a flag complex, the homotopy theory of $\Omega \ux^\K$ and the algebra $H_*(\Omega\ux^\K)$ are well understood. Decompositions up to homotopy for certain spaces $X_i$ were given in \cite{S}, a presentation and an additive description of the loop homology was given in \cite{cai}. For arbitrary $\K$,
information about the torsion appearing in the loop homology and the Steenrod algebra action was given in \cite{SV}, as well as a formula for the Poincar\'e series in terms of $H_*(\Omega\ZK)$. 
However, not much is known about the algebra structure. 

In this paper, we expand out understanding of the algebra structure of $H_*(\Omega \ux^\K)$ by giving a general colimit decomposition of $H_*(\Omega \ux^\K)$ for any simplicial complex $\K$, generalising a result of \cite{Dob09,cai}. For a simplicial complex $\K$, let $\Kf$ be the flag complex on the $1$-skeleton of $\K$. It was shown in \cite{S1} that the homotopy type of $\Omega \ux^\K$ is essentially determined by the homotopy type of $\Omega \ux^{\K_I}$, for all $I \in \Kf$. We prove an algebraic analogue of this result, which gives a colimit decomposition of $H_*(\Omega \ux^\K)$ over all faces of $\Kf$. \begin{thm}[Theorem \ref{thm:uxk_colimit_decomposition} in the text]
\label{thm:uxk_colimit_decomposition_intro}
Let $\k$ be a principal ideal domain and let $X_1,\dots,X_m$ be simply connected pointed spaces. Let $\K$ be a simplicial complex, and suppose $H_*(\Omega\ux^\K;\k)$ is a free $\k$-module. Then
the natural map of Hopf algebras
$$\colim_{I\in\Kf}H_*(\Omega\ux^{\K_I};\k)\to H_*(\Omega\ux^\K;\k)$$
is an isomorphism.
\end{thm}
Here the colimit is taken in the category of graded Hopf algebras (equivalently graded associative algebras with unit, see Lemma \ref{lmm:hopf-colimit}). The proof of Theorem \ref{thm:uxk_colimit_decomposition_intro} is obtained from a more general result showing that the loop homology functor preserves pushouts of maps admitting a left homotopy inverse (see Theorem \ref{thm:general-theorem-for-pushout}).

Since $I\in\Kf$ if and only if $\K_I$ is a \emph{1-neighbourly} complex (i.e. a complex where any two vertices are connected by an edge), Theorem \ref{thm:uxk_colimit_decomposition_intro}  effectively reduces the study of presentations of the algebra $H_*(\Omega\ux^\K)$ to the case when $\K$ is 1-neighbourly. We obtain a similar result about the additive structure. For a graded vector space $V= \oplus_{i \geq 0}V_i$, the \emph{Poincar\'e series} of $V$ is defined as the power series $P(V)=P(V;t):= \sum_{i \geq 0} \mathrm{dim}(V_i)t^i \in \mathbb{Z}[[t]]$. We prove the following.
\begin{thm}[Theorem \ref{thm:poincare_series_reduction_to_neighbourly} in the text]
\label{thm:intropoincare_series_reduction_to_neighbourly}
Let $\k$ be a principal ideal domain and let $X_1,\dots,X_m$ be simply connected pointed spaces of finite type. Let $\K$ be a simplicial complex, and suppose $H_*(\Omega\ux^\K;\k)$ is a free $\k$-module. Then we have the following identity of Poincar\'e series.
\begin{align*}
    1/P(H_*(\Omega\ux^\K;\k))&=\sum_{I\in\Kf} (1-\chi(\lk_{\Kf} I))/P(H_*(\Omega\ux^{\K_I};\k))\\
    &=\sum_{B\in\Kf}(-1)^{|B|}\sum_{I\subseteq B}(-1)^{|I|}/P(H_*(\Omega\ux^{\K_I};\k)).
\end{align*}
\end{thm} 

In Section \ref{sec:Polyhedralproducts}, we apply this to calculate the Poincar\'e series of various families of simplicial complexes, including the $k$-skeleta of flag complexes (which includes any $1$-dimensional simplicial complex), and $2$-dimensional simplicial complexes. In \cite{S} and \cite{S1} respectively, it was shown that when each $X_i$ is a sphere, $\Omega \ux^{\K}$ decomposes as a product of spheres, loops on spheres and certain well-studied torsion spaces up to homotopy. This informs on properties of the loop homology of such polyhedral products as a module. The calculation of Poincar\'e series allows one to enumerate the terms that appear in such a decomposition. 

In the case of the Davis--Januskiewicz space and a field $\k$, the \emph{Stanley--Reisner ring} associated to $\K$ is the ring $\k[\K] := \k[v_1,\cdots,v_m]/(v^I = 0, \: I \notin \K)$. By \cite[Proposition 6.5]{franz-hga}, there is an isomorphism of Hopf algebras $H_*(\Omega \djk;\k) \cong \Ext_{\k[\K]}(\k,\k)$. The Poincar\'e series of $\Ext_{\k[\K]}(\k,\k)$ was expressed in terms of the Poincar\'e series of a finite set of simplicial complexes by Berglund \cite{berglund}. Theorem \ref{thm:uxk_colimit_decomposition_intro} gives a colimit decomposition of such Ext algebras (Corollary \ref{cor:ext}), and Theorem \ref{thm:intropoincare_series_reduction_to_neighbourly} can be used to prove a more explicit formula for the Poincar\'e series in certain cases. We also show that the algebra $\Ext_{\k[\K]}(\k,\k)$ is multiplicatively generated by $\Ext^1$ and $\Ext^2$ for a wide class of complexes (Theorem \ref{thm:new-k2}), extending a result of Conner and Shelton \cite{conner-shelton}.

In addition, we give explicit presentations of $H_*(\Omega \djk;\k)$ as a graded Hopf algebra when $\K$ is a \emph{HMF-presented} simplicial complex (see Definition \ref{defn:HMF-presented}). The class of HMF-presented complexes are closely connected to homology missing face complexes introduced by Beben and Grbi\'c \cite{bg} and totally homology fillable complexes introduced by Iriye and Kishimoto \cite{fat-wedge}. 
Theorem \ref{thm:uxk_colimit_decomposition_intro} allows us to obtain a presentation of $H_*(\Omega\djk;\k)$ under a weaker assumption that $\K_I$ is HMF-presented for each $I\in\Kf$ (see Theorem \ref{thm:hmf-presented-presentation-upgrade}). This assumption holds for $k$-skeleta of a flag simplicial complexes,  so we upgrade the module information of $H_*(\Omega \djk;\k)$ from \cite{S} to a presentation of $H_*(\Omega \djk;\k)$ as an algebra in this case.

Finally we extend the aforementioned results to \emph{weighted polyhedral products}, a generalisation of polyhedral products defined by So, Stanley and Theriault \cite{weighted}. In this context, we give a presentation of certain weighted polyhedral products associated to graphs, and obtain a loop space decomposition in the case of a tree.

In Section \ref{sec:prelim}, we state some algebraic and categorical results we will require, as well as defining the weighted polyhedral product. In Section \ref{sec:loophompushout}, we prove the main technical result behind Theorem \ref{thm:uxk_colimit_decomposition_intro}, namely that the loop homology functor preserves pushouts of maps admitting left homotopy inverses. In Section \ref{sec:presforDJK}, we define HMF-presentable complexes and prove some general properties of the associated $H_*(\Omega \djk)$. Section \ref{sec:Polyhedralproducts} contains the proof of Theorem \ref{thm:uxk_colimit_decomposition_intro}, and this theorem is applied to give presentations of $H_*(\Omega \djk)$ in certain cases. Applications to the Poincar\'e series are given in Section \ref{sec:PS} and generalisations to weighted polyhedral products associated to graphs are considered in Section \ref{sec:weighted}.

The first author was supported by the EPSRC grant EP/Z534894/1 during the preparation of this work. The second author was supported by the Theoretical Physics and Mathematics Advancement Foundation ``BASIS''. The authors would like to thank William Hornslien for helpful discussions in the course of this work.

\section{Preliminary material}
\label{sec:prelim}

\subsection*{Hopf algebras and loop homology}

Let $X$ be a simply connected $CW$-complex and let $\k$ be a commutative ring with unit. Then $H_*(\Omega X;\k)$ is a connected associative algebra, where the multiplication is induced by the loop multiplication $\mu:\Omega X \times \Omega X \rightarrow \Omega X$. If $H_*(\Omega X;\k)$ is a free $\k$-module (e.g. if $\k$ is a field), this algebra has the structure of a cocommutative Hopf algebra, where the coproduct is induced by the diagonal map $\Delta: \Omega X \rightarrow \Omega X \times \Omega X$. Moreover, any map $f:X \rightarrow Y$ of topological spaces induces a map of algebras $(\Omega f)_*:H_*(\Omega X;\K)\to H_*(\Omega Y;\k)$ (a map of Hopf algebras in the torsion-free case).
In general, the functor $H\Omega:Top_* \rightarrow HopfAlg_\k$ does not preserve colimits. However, we show in Section ~\ref{sec:loophompushout} that it does preserve pushouts under certain conditions.  

\begin{prp}
\label{prp:split_hopf_properties}
Consider a homotopy fibration of simply connected spaces which has a homotopy section,
$$\xymatrix{
F\ar[r]^-i
&
E
\ar@{->>}[r]_-p 
&
B,
\ar@{_(->}@<-1ex>[l]_-s
}\quad p\circ s\simeq \id_B.$$
Then the continuous maps
\begin{gather*}
\varphi_L:\Omega F\times\Omega B\overset{\Omega i\times\Omega s}\longrightarrow \Omega E\times\Omega E\overset{\mu}\longrightarrow \Omega E,\quad \varphi_R:\Omega B\times\Omega F\overset{\Omega s\times\Omega i}\longrightarrow \Omega E\times\Omega E\overset{\mu}\longrightarrow \Omega E\end{gather*}
are homotopy equivalences.
In particular, for the induced diagram of algebras and their homomorphisms,
$$\xymatrix{
H_*(\Omega F;\k)\ar[r]^-\iota
&
H_*(\Omega E;\k)
\ar@{->>}[r]_-\pi
&
H_*(\Omega B;\k),
\ar@{_(->}@<-1ex>[l]_-\sigma
}$$
the $\k$-linear maps
$$\Phi_L:H_*(\Omega F;\k)\otimes H_*(\Omega B;\k)\overset{\iota\otimes\sigma}\longrightarrow H_*(\Omega E;\k)\otimes H_*(\Omega E;\k)\overset{\mu}\longrightarrow H_*(\Omega E;\k),$$
$$\Phi_R:H_*(\Omega B;\k)\otimes H_*(\Omega F;\k)\overset{\sigma\otimes\iota}\longrightarrow H_*(\Omega E;\k)\otimes H_*(\Omega E;\k)\overset{\mu}\longrightarrow H_*(\Omega E;\k)$$
are isomorphisms of $\k$-modules if $H_*(\Omega F;\k)$ and $H_*(\Omega B;\k)$ are free $\k$-modules of finite type.
\end{prp}
\begin{proof}
See \cite[Lemma B.2(1) and Theorem B.3(1)]{vylegzhanin25} for $\varphi_L$ and $\Phi_L$; the case of $\varphi_R$ and $\Phi_R$ is similar.
\end{proof}
\begin{rmk}
More generally, if the map $\Omega p: \Omega E\to \Omega B$ has a homotopy section $s':\Omega B\to \Omega E$, then there is a homotopy equivalence $\Omega F\times\Omega B\simeq \Omega E$ and also an \emph{extension of Hopf algebras} $H_*(\Omega F)\to H_*(\Omega E)\to H_*(\Omega B)$. In particular, as in \cite[Lemma B.2(1) and Theorem B.3(1)]{vylegzhanin25} for example, there is an isomorphism $\Phi_L:H_*(\Omega F)\otimes H_*(\Omega B)\overset\simeq\longrightarrow H_*(\Omega E)$ of coalgebras, left $H_*(\Omega F)$-modules and right $H_*(\Omega B)$-comodules. 

Under the assumptions of Proposition \ref{prp:split_hopf_properties}, one can take $s'=\Omega s$, and the corresponding extension of Hopf algebras is ``semi-direct'' (or, following \cite{smith_semi-tensor}, a \emph{semi-tensor product}).
\end{rmk}

Another result, which is important in the proof of Theorem \ref{thm:general-theorem-for-pushout}, is the description of the loop homology of a wedge.
 \begin{thm}
\label{thm:wedge_loop_homology}
Let $X_1$, $X_2$ be simply connected CW-complexes such that each $H_*(\Omega X_i;\k)$ is a free $\k$-module. Then the natural map $H_*(\Omega X_1;\k)\ast H_*(\Omega X_2;\k)\to H_*(\Omega(X_1\vee X_2);\k)$ is an isomorphism.
\end{thm}
\begin{proof}
By \cite[Theorem 2.1]{adams-hilton}, each simply connected CW-complex $X$ admits an \emph{Adams--Hilton model} $AH(X)$ -- a dg-algebra such that the underlying algebra is a free tensor algebra with generators corresponding to cells of $X$, and $H(AH(X))\cong H_*(\Omega X;\k)$. By \cite[Corollary 2.2]{adams-hilton}, $AH(X_1)\ast AH(X_2)$ is an Adams--Hilton model for $X_1\vee X_2$. It follows that there are isomorphisms $H_*(\Omega(X_1\vee X_2);\k)\cong H(AH(X_1)\ast AH(X_2))\cong H(AH(X_1))\ast H(AH(X_2))$, since homology respects free products of dg algebras with torsion-free homology.
 \end{proof}
\begin{rmk}
More generally, Li Cai proved in \cite[Corollary 4.3]{cai} that $H_*(\Omega\ux^\K;\k)$ is a \emph{graph product} of algebras $H_*(\Omega X_i;\k)$, if $\k$ is a field, $X_i$ are simply connected spaces and $\K$ is a flag simplicial complex. Theorem \ref{thm:uxk_colimit_decomposition_intro} is a generalisation of Cai's result.
\end{rmk}

Finally, we require the following result of Ganea \cite{G} and a consequence for the Poincar\'e series of the loop homology of a wedge of spaces.
\begin{lmm}
\label{lmm:ganea}
For pointed spaces $X_1,X_2$, there is a homotopy equivalence \[\Omega(X_1\vee X_2)\simeq\Omega X_1\times\Omega X_2\times\Omega\Sigma(\Omega X_1\wedge\Omega X_2).\qed\]
\end{lmm}

\begin{lmm}
\label{lmm:poincare-series-for-wedge}
Let $X_1$, $X_2$ be simply connected CW-complexes such that each $H_*(\Omega X_i;\k)$ is a free $\k$-module. Then $1/P(H_*(\Omega(X_1\vee X_2);\k)=1/P(H_*(\Omega X _1;\k))+1/P(H_*(\Omega X_2;\k))-1$.
\end{lmm}
\begin{proof}
This follows from Theorem \ref{thm:wedge_loop_homology}, since the identities $P(V\oplus V')=P(V)+P(V')$, and $P(V\otimes V')=P(V)\cdot P(V')$ imply that $1/P(A_1\ast A_2)=1/P(A_1)+1/P(A_2)-1$.

Alternatively, this follows from Lemma \ref{lmm:ganea}, the K\"unneth isomorphism and the Bott--Samelson theorem.
\end{proof}

\subsection*{Categorical Preliminaries}

In this section, we state categorical results that we will require. These are well known to experts, but will be helpful to collect for later reference. 

\begin{lmm}
\label{lmm:hopf-colimit}
If $\{A_\alpha\}$ is a diagram of Hopf algebras, then the colimit $\colim_{\alpha} A_\alpha$ in the category of Hopf algebras is isomorphic to the colimit in the category of associative algebras.
\end{lmm}
\begin{proof}
The forgetful functor $HopfAlg_\k\to Alg_\k$ has a right adjoint by \cite[Theorem 2.5]{agore}, so it preserves colimits.
\end{proof}

Let $\mathcal{C}$ be a small, cocomplete category, and let $D:\mathcal{I}\rightarrow \mathcal{C}$ be a diagram. Let $\mathcal{I}_1$ and $\mathcal{I}_2$ be full subcategories of $\mathcal{I}$, such that $\mathcal{I} = \mathcal{I}_1 \cup \mathcal{I}_2$. Define $\mathcal{I}_L = \mathcal{I}_1 \cap \mathcal{I}_2$. The restriction of $D$ to $\mathcal{I}_1$, $\mathcal{I}_2$ and $\mathcal{I}_L$ define diagrams $D_1$, $D_2$ and $D_L$ respectively. Denote by $D'$ the diagram $D_1 \xleftarrow{i_1} D_L \xrightarrow{i_2} D_2$. The following result follows from \cite[Chapter X, Section 7, Theorem 1]{Mac}

\begin{lmm}
\label{lem:Kanext}
The diagram $D'$ is a left Kan extension of $D$. Hence their colimits are equal. \qed
\end{lmm}

Now we turn our attention to the colimit of $D'$. The maps $i_1$ and $i_2$ induce maps of colimits $f_1: \colim D_L \rightarrow \colim D_1$ and $f_2:\colim D_L \rightarrow D_2$.

\begin{lmm}
\label{lem:pointwisecolimit}
    The colimit of $D'$ is equal to the pushout of $f_1$ and $f_2$.\qed
\end{lmm}

Combining Lemma ~\ref{lem:Kanext} and Lemma ~\ref{lem:pointwisecolimit}, we obtain the following result.

\begin{lmm}
\label{lem:breakdowncolimit}
    Let $\mathcal{C}$ be a small, cocomplete category, and let $D:\mathcal{I}\rightarrow \mathcal{C}$ be a diagram. Let $\mathcal{I}_1$ and $\mathcal{I}_2$ be full subcategories of $\mathcal{I}$, with $\mathcal{I}_L = \mathcal{I}_1 \cap \mathcal{I}_2$. Let $D_1$, $D_2$ and $D_L$ be the restrictions of $D$ to $\mathcal{I}_1$, $\mathcal{I}_2$ and $\mathcal{I}_L$ respectively. Define $f_1:\colim D_L \rightarrow \colim D_1$ and $f_2:\colim D_L \rightarrow \colim D_2$ as the maps induced by the inclusion of diagrams $D_L \rightarrow D_1$ and $D_L \rightarrow D_2$ respectively. Then the colimit of $D$ is equal to the pushout of $f_1$ and $f_2$. \qed
\end{lmm}

\subsection*{Properties of (weighted) polyhedral products}

In this section, we state the definition of a generalisation of polyhedral products called weighted polyhedral products and state some basic properties we will require. The following definition is a special case of \cite[Definition 2.18]{weighted}.

Let $\K$ be a simplicial complex on $[m]$, and let $X_1,\cdots,X_m$ be spaces. For each face $\sigma \in \K$, define $\ux^{\sigma} = \prod\limits_{i\in \sigma} X_i$. Note that for $\tau \subseteq \sigma$, there is a natural inclusion $\ux^{\tau} \rightarrow \ux^{\sigma}$.

Let $Y$ be any space. A \textit{power map} is a collection of maps $\{\rho_a:Y \rightarrow Y\}_{a \in \mathbb{N}}$ such that $\rho_1$ is the identity map, and $\rho_a \circ \rho_b = \rho_{ab}$ for any $a,b \in \mathbb{N}$. For example, if $Y$ is an $H$-space or a co-$H$ space, letting each $\rho_a$ be the degree $a$ map is a power map. If each $X_i$ has an associated power map, for each $\underline{a} = (a_1,\cdots,a_m) \in \mathbb{N}^m$, there is a map $\underline{a}(\sigma):\ux^{\sigma} \rightarrow \ux^{\sigma}$ defined as the product map $\rho_{a_1} \times \cdots \times \rho_{a_m}$.

We wish to restrict to certain sequences $\underline{a}$. Let $\Delta^{m-1}$ be the $(m-1)$-simplex. A \textit{power sequence} is a map $c:\Delta^{m-1} \rightarrow \mathbb{N}^m$, $c(\sigma) = (c_1^{\sigma},\cdots,c_m^{\sigma})$ such that $c_i^{\sigma} = 1$ for $i \notin \sigma$, and $c^{\tau}_i$ divides $c_i^{\sigma}$ for $\tau \subseteq \sigma$ and $1 \leq i \leq m$. In most cases, we will consider a power sequence restricted to a simplicial complex $\K \subseteq \Delta^{m-1}$.

If $\tau$ and $\sigma$ are faces of $\Delta^{m-1}$ such that $\tau \subseteq \sigma$, then $c^{\sigma}_{i}/c^{\tau}_i$ is a positive integer for each $i$. Let $\underline{c}^{\sigma/\tau}$ be the power sequence $(c_1^{\sigma}/c^{\tau}_1,\cdots,c^{\sigma}_m/c^{\tau}_m)$.

\begin{dfn}
    Let $\K$ be a simplicial complex on $[m]$, and let $(X_1,\rho^1),\cdots,(X_m,\rho^m)$ be spaces with associated power maps. Let $\underline{c}:\K \rightarrow \mathbb{N}^m$ be a power sequence. Define the \textit{weighted polyhedral product} as the coequaliser \[\ux^{\K,c} =\colim \bigsqcup\limits_{\tau \rightarrow \sigma \in \K}\ux^{\tau}_{\tau \rightarrow \sigma} \rightrightarrows \bigsqcup\limits_{\sigma \in \K} \ux^{\sigma}_{\sigma},\] where $\ux^{\tau}_{\tau \rightarrow \sigma} = \ux^{\tau}$, the top map is the inclusion $\ux^{\tau}_{\tau \rightarrow \sigma} \rightarrow \ux^{\sigma}_{\sigma}$, and the bottom map is $\underline{c}^{\sigma/\tau}:\ux^{\tau}_{\tau\rightarrow\sigma} \rightarrow \ux^{\tau}_{\tau}$.
\end{dfn}
\begin{exm}
    Letting $\rho^i_a$ be the identity map for all $a \in \mathbb{N}$, and $c^\sigma = (1,\cdots,1)$ for all $\sigma$, we obtain the ordinary polyhedral product $\ux^\K$. 
\end{exm}

Observe that for a simplicial inclusion $\L \rightarrow \K$, there is an inclusion of diagrams, which induces a map of weighted polyhedral products $\ux^{\mathcal{L},c} \rightarrow \ux^{\K,c}$. The following is a weighted polyhedral product analogue of \cite[Proposition 3.1]{GT}. 

\begin{lmm}
\label{lem:pushout}
    Suppose $\K = \K_1 \cup_\L \K_2$. Then there is a pushout of weighted polyhedral products \[\begin{tikzcd}
	{\ux^{\L,c}} & {\ux^{\K_1,c}} \\
	{\ux^{\K_2,c}} & {\ux^{\K,c}.}
	\arrow[from=1-1, to=1-2]
	\arrow[from=1-1, to=2-1]
	\arrow[from=1-2, to=2-2]
	\arrow[from=2-1, to=2-2]
\end{tikzcd}\]
\end{lmm}
\begin{proof}
    This follows from Lemma ~\ref{lem:breakdowncolimit} and the definition of the weighted polyhedral product as a colimit.
\end{proof}

We require one further property of weighted polyhedral products. The following combines \cite[Proposition 5.5, Corollary 5.6]{weighted}.

\begin{lmm}
\label{lem:retraction}
Let $\K$ be a simplicial complex, and $\L$ a full subcomplex of $\K$. Then there is a natural map $r:\ux^{\K,c} \rightarrow \ux^{\L,c}$ which is a left homotopy inverse for the map $\ux^{\L,c} \rightarrow \ux^{\K,c}$ induced by the inclusion $\L \rightarrow \K$. \qed
\end{lmm}

Now suppose that a simplicial complex $\K$ satisfies $\K=\K_A\cup\K_B$ for some $A,B\subsetneq[m]$. By Lemma \ref{lem:pushout} and Lemma \ref{lem:retraction}, we obtain a homotopy pushout 
\[\begin{tikzcd}
	{\ux^{\K_{A\cap B},c}} & {\ux^{\K_A,c}} \\
	{\ux^{\K_B,c}} & {\ux^{\K,c},}
	\arrow["{f_A}", from=1-1, to=1-2]
	\arrow["{f_B}", from=1-1, to=2-1]
	\arrow[from=1-2, to=2-2]
	\arrow[from=2-1, to=2-2]
\end{tikzcd}\]
where the maps $f_A$ and $f_B$ admit left homotopy inverses. In the next section, we study how the loop space functor interacts with pushout of this form.

\section{Loop homology of a pushout of retractions}
\label{sec:loophompushout}

In this section, we prove that the loop homology functor preserves pushouts of maps with left homotopy inverses. We will assume throughout this section that homology is taken to have coefficients in some commutative ring with unit. Moreover, all spaces will be assumed to have homology which are free modules over this ring.

\subsection*{Homotopy theoretic setup}

Suppose we have a homotopy pushout
\[\begin{tikzcd}
	{X_{AB}} & {X_A} \\
	{X_B} & X
	\arrow["{f_A}", from=1-1, to=1-2]
	\arrow["{f_B}", from=1-1, to=2-1]
	\arrow["{g_A}", from=1-2, to=2-2]
	\arrow["{g_B}", from=2-1, to=2-2]
    \arrow["{f}", from=1-1, to=2-2]
\end{tikzcd}\] of simply connected $CW$-complexes
such that the maps $f_A$ and $f_B$ have left homotopy inverses $r_A$ and $r_B$, respectively. Then $f$ admits a left homotopy inverse $r$ by the universal property, see the diagram below:
\begin{equation}\label{eqn:spacepushout}\begin{tikzcd}
	{X_{AB}} & {X_A} & \\
	{X_B} & X \\
	&& {X_{AB}.}
	\arrow["{f_A}", from=1-1, to=1-2]
	\arrow["{f_B}", from=1-1, to=2-1]
	\arrow["{g_A}", from=1-2, to=2-2]
	\arrow["{r_A}", curve={height=-12pt}, from=1-2, to=3-3]
	\arrow["{g_B}"', from=2-1, to=2-2]
	\arrow["{r_B}", curve={height=12pt}, from=2-1, to=3-3]
	\arrow["r", from=2-2, to=3-3]
\end{tikzcd}\end{equation}

By construction, the composite map $X_{AB}\to X_{AB}$ is homotopic to the identity.
Define $F$, $F_A$ and $F_B$ as the homotopy fibres of $r$, $r_A$ and $r_B$. Arguing as in \cite[Proposition 3.3]{theriault}, we obtain a homotopy pushout \[\begin{tikzcd}
	{*} & {F_A} \\
	{F_B} & F.
	\arrow[from=1-1, to=1-2]
	\arrow[from=1-1, to=2-1]
	\arrow[from=1-2, to=2-2]
	\arrow[from=2-1, to=2-2]
\end{tikzcd}\] We immediately obtain the following result.
\begin{lmm}
\label{lmm:use of cube lemma}
The natural map $F_A\vee F_B\to F$ is a homotopy equivalence. \qed
\end{lmm}
By definition of $r$, there is a homotopy fibration diagram 
\begin{equation}
\label{eqn:maps-of-fibrations} 
    \begin{tikzcd}
	{F_A} & {X_A} & {X_{AB}} \\
	F & X & {X_{AB}} \\
	{F_B} & {X_B} & {X_{AB}.}
	\arrow[from=1-1, to=1-2]
	\arrow[from=1-1, to=2-1]
	\arrow["{{r_A}}", from=1-2, to=1-3]
	\arrow["{{g_A}}", from=1-2, to=2-2]
	\arrow[equals, from=1-3, to=2-3]
	\arrow[from=2-1, to=2-2]
	\arrow["r", from=2-2, to=2-3]
	\arrow[equals, from=2-3, to=3-3]
	\arrow[from=3-1, to=2-1]
	\arrow[from=3-1, to=3-2]
	\arrow["{{g_B}}"', from=3-2, to=2-2]
	\arrow["{{r_B}}", from=3-2, to=3-3]
\end{tikzcd}
\end{equation}

\begin{lmm}
\label{lmm:split-pushout-homotopy-equivalences}
There are homotopy equivalences $\Omega X_A\simeq\Omega F_A\times\Omega X_{AB}$, $\Omega X_B\simeq\Omega F_B\times\Omega X_{AB}$, $\Omega X\simeq\Omega X_{AB}\times\Omega F_A\times\Omega F_B\times\Omega\Sigma(\Omega F_A\wedge\Omega F_B)$.
\end{lmm}
\begin{proof}
In \eqref{eqn:maps-of-fibrations}, the rows are homotopy fibrations with homotopy sections $f_A$, $f_B$, $f$, respectively. Hence the homotopy equivalences follow from Proposition \ref{prp:split_hopf_properties}, Lemma \ref{lmm:use of cube lemma} and Lemma \ref{lmm:ganea}.
\end{proof}

\subsection*{Algebraic setup}
We denote $H_{AB}:=H_*(\Omega X_{AB}),$
\begin{gather*}
H_A:=H_*(\Omega X_A),\quad H_B:=H_*(\Omega X_B),\quad H:=H_*(\Omega X),\\
S_A:=H_*(\Omega F_A),\quad S_B:=H_*(\Omega F_B),\quad S:=H_*(\Omega F).
\end{gather*}
The homotopy pushout \eqref{eqn:spacepushout} induces the commutative diagram
\begin{equation}\begin{tikzcd}
	{H_{AB}} & {H_A} & \\
	{H_B} & H \\
	&& {H_{AB},}
	\arrow["{s_A}", from=1-1, to=1-2]
	\arrow["{s_B}", from=1-1, to=2-1]
	\arrow["{j_A}", from=1-2, to=2-2]
	\arrow["{p_A}", curve={height=-12pt}, from=1-2, to=3-3]
	\arrow["{j_B}", from=2-1, to=2-2]
	\arrow["{p_B}", curve={height=12pt}, from=2-1, to=3-3]
	\arrow["p", from=2-2, to=3-3]
    \arrow["{s}", from=1-1, to=2-2]
\end{tikzcd}\end{equation}
where $p\circ s:H_{AB}\to H_{AB}$ is the identity map. The diagram \eqref{eqn:maps-of-fibrations} induces the commutative diagram
$$\xymatrix{
S_A\ar[r]^-{i_A}
\ar[d]
&
H_A
\ar[r]^-{p_A}
\ar[d]^-{j_A}
&
H_{AB}
\ar@{=}[d]
\\
S
\ar[r]^-i
&
H\ar[r]^-p
&
H_{AB}
\ar@{=}[d]
\\
S_B
\ar[r]^-{i_B}
\ar[u]
&
H_B
\ar[u]_-{j_B}
\ar[r]^-{p_B}
&
H_{AB}.
}$$
\begin{lmm}
\label{lmm:Phi_are_iso}
The following $\k$-linear maps are isomorphisms:
\begin{enumerate}
    \item $S_A\otimes H_{AB}\to H_A,$ $a\otimes h\mapsto i_A(s)\cdot s_A(h)$ and $
H_{AB}\otimes S_A\to H_A,$ $h\otimes a\mapsto s_A(h)\cdot i_A(s)$.
\item $S_B\otimes H_{AB}\to H_B,$ $b\otimes h\mapsto i_B(s)\cdot s_B(h)$ and $H_{AB}\otimes S_B\to H_B,$ $h\otimes b\mapsto s_B(h)\cdot i_B(s)$.
\item $S\otimes H_{AB}\to H$, $x\otimes h\mapsto i(x)\cdot s(h)$.
\end{enumerate}
\end{lmm}
\begin{proof}
The homotopy fibrations $F_A\to X_A\xrightarrow{r_A} X_{AB}$, $F_B\to X_B\xrightarrow{r_B} X_{AB}$ and $F\to X\xrightarrow{r} X_{AB}$ have homotopy sections $f_A$, $f_B$, $f$, respectively. The result then follows from Proposition \ref{prp:split_hopf_properties}.
\end{proof}
\begin{lmm}
\label{lmm:use of cube lemma alg}    
The natural map $S_A\ast S_B\to S$ is an isomorphism.
\end{lmm}
\begin{proof}
By definition, the map in the statement of the lemma equals the composite
$$S_A\ast S_B=H_*(\Omega F_A)\ast H_*(\Omega F_B)\to H_*(\Omega(F_A\vee F_B))\to H_*(\Omega F)=S.$$
The first map is an isomorphism by Theorem \ref{thm:wedge_loop_homology}, and the second map is an isomorphism since $F_A\vee F_B\to F$ is a homotopy equivalence by Lemma \ref{lmm:use of cube lemma}.
\end{proof}

Consider the pushout $\widehat{H}:=H_A\ast_{H_{AB}}H_B$ of the injective maps $s_A$ and $s_B$. Let $\widehat{s}$, $\widehat{j}_A$, $\widehat{j}_B$ be the induced maps to $\widehat{H}$ as in the diagram below
$$\xymatrix{
&
&S_A
\ar[d]^-{i_A}\\
&H_{AB}
\ar[r]^-{s_A}
\ar[d]^-{s_B}
\ar[rd]^-{\widehat{s}}
&
H_A\ar[d]^-{\widehat{j}_A}
\\
S_B
\ar[r]^-{i_B}
&
H_B
\ar[r]^-{\widehat{j}_B}
&
\widehat{H}.
}$$

For homomorphisms of algebras $h_1:R_1\to S_1$, $h_2:R_2\to S_2$, we denote the componentwise map by $h_1\ast h_2:R_1\ast R_2\to S_1\ast S_2$. On the other hand, for $h_1:R_1\to S$ and $h_2:R_2\to S$ we denote by $h_1\sqcup h_2:R_1\ast R_2\to S$ the map of algebras defined by the universal property. It follows by definition that
$h_1\sqcup h_2=(\id_S\sqcup\id_S)\circ (h_1\ast h_2):$ $R_1\ast R_2\to S.$
\begin{lmm}
\label{lmm:lambda surjective}
Consider the $\k$-linear map $\lambda:(S_A\ast S_B)\otimes H_{AB}\to \widehat{H}$ defined as the composition
$$\lambda:
(S_A\ast S_B)\otimes H_{AB}\overset{(i_A\ast i_B)\otimes \id}\longrightarrow (H_A\ast H_B)\otimes H_{AB}\overset{(\widehat{j}_A\sqcup \widehat{j}_B)\otimes \widehat{s}}\longrightarrow\widehat{H}\otimes\widehat{H}\overset{\mu}\longrightarrow\widehat{H}.$$
Then
\begin{enumerate}
    \item $\Img\lambda\subseteq\widehat{H}$ is a subalgebra;
    \item $\lambda$ is a surjective map.
\end{enumerate} 
\end{lmm}
\begin{proof}
Denote $\overline{a}:=\widehat{j}_A(i_A(a))$ for $a\in S_A$, $\overline{b}:=\widehat{j}_B(i_B(b))$ for $b\in S_B$, $\overline{h}:=\widehat{s}(h)$ for $h\in H_{AB}$. By definition, we have
$$\lambda(a_1b_1\dots a_Nb_N\otimes h):=\overline{a}_1\cdot \overline{b}_1\dots\overline{a}_N\cdot\overline{b}_N\cdot\overline{h}\in\widehat{H},$$
hence $\Img\lambda\subseteq\widehat{H}$ is additively generated by elements of that form.

Let $a\in S_A$ and $h\in H_{AB}$.
By Lemma \ref{lmm:Phi_are_iso}, we have
$i_A(a)\cdot s_A(h)=\sum_\alpha s_A(h_\alpha)\cdot i_A(a_\alpha)\in H_A$
for some $a_\alpha\in S_A$, $h_\alpha\in H_{AB}$. Applying $\widehat{j}_A$ to both sides, we obtain
$\overline{a}\cdot\overline{h}=\sum_\alpha\overline{h}_\alpha\cdot\overline{a}_\alpha\in  \widehat{H}.$ Similarly, $\overline{b}\cdot\overline{h}=\sum_\beta\overline{h_\beta}\cdot\overline{b_\beta}$ for some $b_\beta\in S_B$, $h_\beta\in H_{AB}$.

By these formulas, any element of the form $\overline{h}\cdot \overline{a}_1\overline{b}_1\dots \overline{a}_N\overline{b}_N$ is a linear combination of elements of the form $\overline{a}_1'\overline{b}_1'\dots\overline{a}_N'\overline{b}_N'\overline{h}'$. Also, $\widehat{s}$ is a map of algebras, so $\overline{h}_1\cdot\overline{h}_2=\overline{h_1\cdot h_2}$. It follows that the product of two elements of the form
$\overline{a}_1\overline{b}_1\dots\overline{a_N}\overline{b_N}\cdot\overline{h}$
is a linear combination of elements of the same form. This proves (1).

To prove (2), note that $\Img(\widehat{j}_A)\cup\Img(\widehat{j}_B)$ is a generating set for the algebra $\widehat{H}=H_A\ast_{H_{AB}}H_B$. Since $\Img\lambda\subseteq \widehat{H}$ is a subalgebra by (1), it is sufficient to show that $\Img(\widehat{j}_A)\cup\Img(\widehat{j}_B)\subseteq\Img\lambda$. 
By part (3) of Lemma \ref{lmm:Phi_are_iso}, we have an additive isomorphism
$$\Phi_L:S_A\otimes H_{AB}\to H_A,\quad s\otimes h\mapsto i_A(s)\cdot s_A(h),$$
hence the $\k$-module $\Img(\widehat{j}_A:H_A\to\widehat{H})$ is generated by elements of the form $$\widehat{j}_A(i_A(s)\cdot s_A(h))=\widehat{j}_A(i_A(s))\cdot \widehat{j}_A(s_A(h))=\overline{a}\cdot\overline{h}\in\Img\lambda.$$
It follows that $\Img\widehat{j}_A\subseteq \Img\lambda$.
Similarly, $\Img\widehat{j}_B\subseteq\Img\lambda$.
\end{proof}

Define the map of algebras $c:\widehat{H}\to H$, using the universal property of a pushout as in the diagram below:
$$\xymatrix{
H_{AB}
\ar[r]^-{s_A}
\ar[d]_-{s_B}
\ar[rd]^-{\widehat{s}}
&
H_A\ar[d]^-{\widehat{j}_A}
\ar@/^/[rdd]^-{j_A}
\\
H_B
\ar[r]^-{\widehat{j}_B}
\ar@/_/[rrd]_-{j_B}
&
\widehat{H}
\ar@{-->}[rd]^-c
\\
&&H.
}$$
Since the diagram commutes, $c\circ\widehat{s}=s:H_{AB}\to H.$

\subsection*{The main theorem}
For a space $X$, denote by $\tau(X)$ the set of primes appearing as $p$-torsion in $H_*(X;\ZZ)$.
\begin{thm}
\label{thm:general-theorem-for-pushout}
Let $X_{AB},X_A,X_B$ be simply connected pointed spaces of finite type, and suppose there are maps $f_A:X_{AB}\to X_A$ and $f_B:X_{AB}\to X_B$ which admit left homotopy inverses. Denote by $X$ the homotopy pushout of $f_A$ and $f_B$. Then
\begin{enumerate}
    \item $\tau(\Omega X)=\tau(\Omega X_A)\cup\tau(\Omega X_B)$.
    \item If $H_*(\Omega X_A;\k)$, $H_*(\Omega X_B;\k)$ and $H_*(\Omega X_{AB};\k)$ are free $\k$-modules, then the natural map of Hopf algebras \[H_*(\Omega X_A;\k) *_{H_*(\Omega X_{AB};\k)} H_*(\Omega X_B;\k)\to H_*(\Omega(X_A\cup_{X_{AB}}X_B);\k)\]
is an isomorphism.
    \item Under the assumptions of (2), $$1/P(H_*(\Omega X;\k))=1/P(H_*(\Omega X_A;\k))+1/P(H_*(\Omega X_B;\k))-1/P(H_*(\Omega X_{AB};\k)).$$
\end{enumerate}
\end{thm}
\begin{proof}
For (1), note that $\tau(Y\times Z)=\tau(Y)\cup\tau(Z)$ by K\"unneth theorem. By Lemma \ref{lmm:split-pushout-homotopy-equivalences}, it is sufficient to prove that $\tau(\Omega\Sigma(\Omega F_A\wedge\Omega F_B))\subseteq\tau(\Omega F_A)\cup\tau(\Omega F_B)$. This follows from \cite[Lemma 6.5]{SV}.

(2) is equivalent to the statement that $c:\widehat{H}\to H$ is an isomorphism.
Consider the diagram
$$\xymatrix{
(S_A\ast S_B)\otimes H_{AB}
\ar[d]^-{(j'_A\sqcup j'_B)\otimes\id}
\ar[rr]^-{(i_A\ast i_B)\otimes\id}
&&
(H_A\ast H_B)\otimes H_{AB}
\ar[d]^-{(j_A\sqcup j_B)\otimes\id}
\ar[rr]^-{(\widehat{j}_A\sqcup\widehat{j}_B)\otimes \widehat{s}}
&&
\widehat{H}\otimes\widehat{H}
\ar[d]^-{c\otimes c}
\ar[r]^-\mu
&
\widehat{H}
\ar[d]^-{c}
\\
S\otimes H_{AB}
\ar[rr]^-{i\otimes \id}
&&
H\otimes H_{AB}
\ar[rr]^-{\id\otimes s}
&&
H\otimes H
\ar[r]^-\mu
&
H.
}$$
 The left square commutes since $i\circ j'_A=j_A\circ i_A$, and $i\circ j'_B=j_B\circ i_B$. The middle square commutes, since $c\circ\widehat{s}=s$, $c\circ\widehat{j}_A=j_A$, and $c\circ\widehat{j}_B=j_B$. The right square commutes since $c$ is a map of algebras. We thus obtain a commutative diagram
$$\xymatrix{
(S_A\ast S_B)\otimes H_{AB}
\ar[d]_-{(j_A'\sqcup j_B')\otimes\id}
\ar[r]^-\lambda
&
\widehat{H}
\ar[d]^-c
\\
S\otimes H_{AB}
\ar[r]^-{\Phi_L}
&
H.
}$$
The left map is an isomorphism by Lemma \ref{lmm:use of cube lemma alg}, and the bottom map is an isomorphism by Lemma \ref{lmm:Phi_are_iso}. Hence the composite map
$$(S_A\ast S_B)\otimes H_{AB}\overset\lambda\longrightarrow \widehat{H}\overset{c}\longrightarrow H$$
is an isomorphism of $\k$-modules. It follows that $\lambda$ is injective. By Lemma \ref{lmm:lambda surjective}, $\lambda$ is surjective, hence it is an isomorphism. Since $c\circ\lambda$ is an isomorphism, so is $c$.

For (3), Lemmas \ref{lmm:Phi_are_iso} and \ref{lmm:use of cube lemma alg} imply that
$$S\simeq S_A\ast S_B,\quad H\simeq S\otimes H_{AB},\quad H_A\simeq S_A\otimes H_{AB},\quad H_B\simeq S_B\otimes H_{AB},$$
hence
$P(S_A)\cdot P(H_{AB})=P(H_A)$, $P(S_B)\cdot P(H_{AB})=P(H_B)$ and $P(H)=P(S_A\ast S_B)\cdot P(H_{AB})$. By Lemma \ref{lmm:poincare-series-for-wedge}, we obtain
$$1/P(H)=1/P(H_{AB})\cdot 1/P(S_A\ast S_B)=1/P(H_A)+1/P(H_B)-1/P(H_{AB}),$$
as required.
\end{proof}

If $X$ is a space defined as in Theorem \ref{thm:general-theorem-for-pushout}, we can give a decomposition of $\Omega X$ under certain conditions. To do this, we require some notation following \cite{S1}. For a collection of topological spaces $\mathcal{X}$, let $\prod \mathcal{X}$ (resp. $\bigvee \mathcal{X}$) be the collection of spaces homotopy equivalent to a finite type product (resp. wedge) of spaces in $\mathcal{X}$. Let $\mathcal{P} := \{S^1,S^3,S^7, \Omega S^n \: : \: n \geq 2, n \notin\{2,4,8\}\}$. By \cite[Theorem 1.1]{HW}, for any finite type $H$-space $X$ localised at a prime $p$, there is a unique decomposition of $X$, up to homotopy, as a finite type product of indecomposable spaces. For $n \geq 2$, $r \geq 1$, and $p$ a prime, let $P^n(p^r)$ be the homotopy cofibre of the degree $p^r$ map on $S^{n-1}$. Let $\mathcal{T}$ be the collection of indecomposable spaces which appear in the decomposition of the loop space of a wedge of Moore spaces of the form $\bigvee_{i=1}^m P^{n_i}(p_i^{r_i})$, where $m \geq 1$, $n_i \geq 3$, $p_i$ is a prime and $r_i \geq 1$. 

The indecomposable spaces which appear in the loop space decomposition of $\bigvee_{i=1}^m P^{n_i}(p_i^{r_i})$, where $m \geq 1$, $n_i \geq 3$, and each $p_i^{r_i}\neq2$, are known and studied in \cite{C,CMN1,CMN2}. If there exists $p_i^{r_i}=2$, these spaces are more mysterious, but partial progress has been made in \cite{W}. 

Now let $\mathcal{W} := \{S^n \: | \: n \geq 2\}$. As in the product case, \cite[Theorem 1.1]{HW} implies that for any finite type co--$H$-space $X$ localised at a prime $p$, there is a unique decomposition of $X$, up to homotopy, as a finite type product of indecomposable spaces. Let $\mathcal{M}$ be the collection of Moore spaces of the form $P^{n}(p^r)$, where $n \geq 3$, $p$ is a prime, and $r \geq 1$, and the indecomposable factors which appear as wedge summands in the unique $2$-local wedge decomposition of spaces of the form $\Sigma ((P^{n_1}(2) \wedge \cdots \wedge (P^{n_l}(2))$, where $l \geq 2$, and each $n_i \geq 3$. 

With this notation, we obtain the following result, the proof of which is similar to \cite[Theorem 5.3]{S1}.
\begin{thm}
\label{thm:decomposition-for-pushout}
Under assumptions of Theorem \ref{thm:general-theorem-for-pushout}, the following hold:
\begin{itemize}
    \item if $\Omega X_A,\Omega X_B\in\prod (\mathcal{P\cup T})$, then $\Omega X\in\prod(\mathcal{P\cup T})$;
    \item if $\Omega X_A,\Omega X_B\in\prod \mathcal{P}$, then $\Omega X\in\prod \mathcal{P}$. \qed
\end{itemize}
\end{thm}
\begin{proof}
    We prove the $\prod(\mathcal{P} \cup \mathcal{T})$ case. The same proof holds when restricting to spaces in $\prod\mathcal{P}$. 
    
Suppose $\Omega X_A,\Omega X_B\in\prod(\mathcal{P\cup T})$. By Lemma \ref{lmm:split-pushout-homotopy-equivalences}, there is a homotopy equivalence \[\Omega X \simeq \Omega X_{AB} \times \Omega F_A \times \Omega F_B \times \Omega \Sigma(\Omega F_A \wedge \Omega F_B),\] where $F_A$, $F_B$ are the homotopy fibres of the retractions $X_A \to X_{AB}$ and $X_B \to X_{AB}$ respectively. By assumption $\Omega X_A,\Omega X_B\in\prod (\mathcal{P\cup T})$, and so the decomposition implies that it suffices to show $\Omega F_A$, $\Omega F_B$, and $\Omega \Sigma(\Omega F_A \wedge \Omega F_B)$ are in $\prod(\mathcal{P} \cup \mathcal{T})$.

By Lemma \ref{lmm:split-pushout-homotopy-equivalences}, $\Omega F_A$ and $\Omega F_B$ retract off $\Omega X_A$ and $\Omega X_B$. By \cite[Theorem 4.9]{S1}, $\prod(\mathcal{P} \cup \mathcal{T})$ is closed under retracts, implying that $\Omega F_A$, $\Omega F_B \in \prod(\mathcal{P}\cup \mathcal{T})$. Hence \cite[Lemma 5.2 (4)]{S1} implies that $\Sigma(\Omega F_A \wedge \Omega F_B) \in \bigvee(\mathcal{W} \cup \mathcal{M})$. The result then follows from \cite[Lemma 5.2 (2)]{S1}.
\end{proof}

\section{Presentations for loop homology of some Davis--Januszkiewicz spaces}
\label{sec:presforDJK}

In this section, we prove some general results about the algebra $H_*(\Omega \djk)$ to set up for the applications of Theorem \ref{thm:uxk_colimit_decomposition_intro}. Our main result is Theorem \ref{thm:hodj-presentation-for-torsionfree-hmf-presented}, which is a generalisation of an unpublished result obtained by Dobrinskaya \cite{dobrinskaya_unpublished}. 

\subsection{Missing faces in simplicial complexes}
Let $\K$ be a simplicial complex on $[m]$ without ghost vertices. For $I\subseteq[m]$, define the simplicial complexes
$$\Delta_I:=\{J:J\subseteq I\},~\partial\Delta_I:=\{J:J\subsetneq I\},~\partial^2\Delta_I:=\{J\subseteq I:~|I\setminus J|\geq 2\}=\sk_{|I|-3}\Delta_I.$$

\begin{dfn}
A subset $I\subseteq[m]$ is a \emph{missing face} of $\K$ if $\K_I=\partial\Delta_I$. Denote by $\MF(\K)$ the set of missing faces, and define $\MF_i(\K):=\{I\in\MF(\K),~|I|=i\}$, and $\MF_{\geq i}(\K)=\{I\in\MF(\K),~|I|\geq i\}$.

A subset $L\subseteq[m]$ is an \emph{almost missing face} of $\K$ if $\partial^2\Delta_L\subseteq\K_L\subsetneq\partial\Delta_L$. Denote by $\AMF(\K)$ the set of almost missing faces.
\end{dfn}
If $A \subseteq \MF(\K)$, then it is easy to check that $\K\cup A$ is a simplicial complex. Also, it is clear that $\MF(\K_J)=\{I\in\MF(\K):~I\subseteq J\}$ and $\AMF(\K_J)=\{I\in\AMF(\K):~I\subseteq J\}$. 

We use the following notation for the augmented chain complex of a simplicial complex $\K$:
$$\widetilde{C}_*(\K;\k)=\bigoplus_{I\in\k}\k\cdot e_I,\quad d(e_I)=\sum_{i\in I}(-1)^{|I_{<i}|}e_{I\setminus\{i\}},\quad |I_{<i}|=\#\{j\in I:~j<i\}.$$
\begin{lmm}
\label{lmm:amf_homology}
For each $J\subseteq[m]$ and $I\in\MF(\K_J)$, there is a well defined class $[\partial\Delta_I]\in\H_{|I|-2}(\K_J;\k)$ repesented by the simplicial chain $c_I=\sum_{i\in I}(-1)^{|I_{<i}|}e_{I\setminus\{i\}}\in\widetilde{C}_{|I|-2}(\K_J;\k)$.

For each $L\in\AMF(\K_J)$, there is an additive relation
\begin{equation}
\label{eqn:amf_homology_relation}
\rho_L=0\in\H_{|L|-3}(\K_J;\k),\quad \rho_L:=\sum_{\begin{smallmatrix}
i\in L:\\L\setminus i\notin\K
\end{smallmatrix}}
(-1)^{|L_{<i}|}[\partial\Delta_{L\setminus i}]
\end{equation}
between these classes. Thus the composition of the natural maps
\begin{equation}
\label{eqn:amf_homology_presentation}
\bigoplus_{L\in\AMF_{i+3}(\K_J)}\k\cdot\rho_L\overset{\nu_J}\longrightarrow \bigoplus_{I\in\MF_{i+2}(\K_J)}\k\cdot x_I\overset{\mu_J}\longrightarrow\H_i(\K_J;\k)
\end{equation}
is zero for each $i\geq 0$.
\end{lmm}
\begin{proof}
Consider $\widetilde{C}_*(\K_J;\k)$ as a chain subcomplex in $\widetilde{C}_*(\Delta_J;\k)$.  For $I\in\MF(\K_J)$, the chain $c_I\in\widetilde{C}_*(\K_J;\k)$ is a boundary $d(e_I)$ in the larger chain complex, so it is a cycle in the smaller complex; it follows that the class $[\partial\Delta_I]$ is well defined.

For $L\in\AMF(\K_J)$, consider the chain $$\phi:=-\sum_{\begin{smallmatrix}i\in L:\\L\setminus i\in\K\end{smallmatrix}}(-1)^{|L_{<i}|}e_{L\setminus i}\in\widetilde{C}_*(\K_J;\k).$$
In the larger chain complex $\widetilde{C}_*(\Delta_J;\k)$, we have
$$d(\phi)=d(d(e_L))+d(\phi)=d(\sum_{\begin{smallmatrix}i\in L:\\L\setminus i\notin\K\end{smallmatrix}}(-1)^{|L_{<i}|}e_{L\setminus i})=\sum_{\begin{smallmatrix}
i\in L:\\L\setminus i\notin\K\end{smallmatrix}}(-1)^{|L_{<i}|}c_{L\setminus i}
.$$
The right hand side is a cycle in $\widetilde{C}_{*}(\K_J;\k)$ which represents the class $\rho_L$, and the left hand side is a boundary in $\widetilde{C}_*(\K_J;\k)$. It follows that $\rho_L=0$.
\end{proof}

\begin{dfn}
\label{defn:HMF-presented}
We say that a simplicial complex $\K$ is
\begin{enumerate}
    \item \emph{homology missing face} (or \emph{HMF}) over $\k$, if for each $J\subseteq[m]$ the group $\H_*(\K_J;\k)$ is generated by the classes of missing faces, i.e. the map $\mu_J$ in the diagram \eqref{eqn:amf_homology_presentation} is surjective.
    \item \emph{HMF-presented} over $\k$ if for each $J\subseteq[m]$, the group $\H_*(\K_J;\k)$ admits the following additive presentation: it is spanned by classes of missing faces modulo the relations \eqref{eqn:amf_homology_relation}, i.e. the diagram \eqref{eqn:amf_homology_presentation} becomes an exact sequence \[\bigoplus_{L\in\AMF(\K_J)}\k\cdot\rho_L\overset{\nu_J}\longrightarrow \bigoplus_{I\in\MF(\K_J)}\k\cdot x_I\overset{\mu_J}\longrightarrow\H_*(\K_J;\k) \to 0.\] 
\end{enumerate}
\end{dfn}
\begin{lmm}
\label{lmm:skeleta-HMF-presented}
For any $0\leq d\leq m-1$, the complex $\sk_d\Delta_{[m]}$ is HMF-presented over any ring $\k$.
\end{lmm}
\begin{proof}
The simplicial complex $\K$
has dimension $d$ and is homotopy equivalent to a wedge of $d$-dimensional spheres, so
$\H_d(\K)\cong\Ker(\widetilde{C}_d(\K)\to\widetilde{C}_{d-1}(\K))=\Ker(\widetilde{C}_d(\Delta_{[m]})\to\widetilde{C}_{d-1}(\Delta_{[m]}))$, and $\H_i(\K)=0$ for $i\neq d$.
Putting this together, we obtain an exact sequence
\[\widetilde{C}_{d+2}(\Delta_{[m]};\k)\to\widetilde{C}_{d+1}(\Delta_{[m]};\k)\to\H_*(\K;\k)\to 0,\]
which is equivalent to \eqref{eqn:amf_homology_presentation}.
\end{proof}

\begin{prp}
\label{prp:HMF-criterion}
Let $\K$ be a simplicial complex and $\k$ be principal ideal domain. Then $\H_*(\K;\k)$ is generated by missing faces if and only if the following conditions hold:
\begin{enumerate}
    \item The abelian group $C:=\Coker(\mu:\bigoplus_{I\in\MF(\K)}\ZZ\cdot x_I\to\H_*(\K;\ZZ))$ is finite;
    \item  $p\in\k^\times$ for each prime $p\in\ZZ$ such that $C$ has $p$-torsion;
    \item  $\chr\k\neq p$ for each prime $p\in\ZZ$ such that $\H_*(\K;\ZZ)$ has $p$-torsion.
\end{enumerate}
Moreover, if this condition holds, then $\H_*(\K;\k)\cong\H_*(\K;\ZZ)\otimes\k$.
\end{prp}
\begin{proof}
The definition of $\mu$ and the Universal Coefficient theorem implies there are exact sequences
\[\begin{tikzcd}
	& {\bigoplus\limits_{I \in \MF(\K)} \k \cdot x_I} & {\tilde{H}_*(\K;\mathbb{Z}) \otimes \k} & {C \otimes \k} & 0 \\
	0 & {\tilde{H}_*(\K;\mathbb{Z}) \otimes \k} & {\tilde{H}_*(\K;\k)} & {\mathrm{Tor}(\tilde{H}_{*-1}(\K;\mathbb{Z}),\k)} & 0.
	\arrow["\mu", from=1-2, to=1-3]
	\arrow[from=1-3, to=1-4]
	\arrow[from=1-4, to=1-5]
	\arrow[from=2-1, to=2-2]
	\arrow[from=2-2, to=2-3]
	\arrow[from=2-3, to=2-4]
	\arrow[from=2-4, to=2-5]
\end{tikzcd}\]
Hence $\H_*(\K;\k)$ is generated by missing faces if and only if $C\otimes\k=0$ and $\Tor(\H_*(\K;\ZZ),\k)=0$. Observe that $\k\otimes\ZZ \cong \k\neq 0$, and $\k\otimes\ZZ/p^i=0$ if and only if $p\in\k^\times$. Since $C$ and $\H_*(\K;\ZZ)$ are finitely generated abelian groups, it follows that  condition $C\otimes\k=0$ is equivalent to ``(1) and (2)''. Moreover, $\Tor(\ZZ,\k)=0$ for any $\k$, and $\Tor(\ZZ/p^i,\k)=0$ if and only if the map $p^i\cdot(-):\k\to\k$ is injective. The latter holds if and only if $p\neq\chr\k$. Hence, $\Tor(\H_*(\K;\ZZ),\k)=0$ is equivalent to (3).
\end{proof}
\begin{crl}
\label{crl:hmf-over-Z}
The following conditions are equivalent:
\begin{enumerate}
    \item $\K$ is HMF over $\ZZ/p$ for all prime $p\in\ZZ$;
    \item $\K$ is HMF over $\ZZ$, and $\H_*(\K_J;\ZZ)$ has no torsion for all $\varnothing\neq J\subseteq[m]$.
\end{enumerate}
\end{crl}
\begin{proof}
By Proposition \ref{prp:HMF-criterion}, both conditions are equivalent to the following: for each $\varnothing\neq J\subseteq[m]$, $C=0$, and $\H_*(\K_J;\ZZ)$ has no torsion.
\end{proof}
\begin{prp}
\label{prp:integer_coefficient_basis}
 Let $\k$ be a principal ideal domain, and suppose that $\H_*(\K;\k)$ is a free $\k$-module generated by missing faces. Then there are integers $\{\mu_{I,q}\in\ZZ:q\in Q,~I\in\MF(\K)\}$ for some finite set $Q$ such that $\{\sum_{I\in\MF(\K)}\mu_{I,q}[\partial\Delta_I]:~q\in Q\}$ is a basis of the $\k$-module $\H_*(\K_J;\k)$.
\end{prp}
\begin{proof}
By assumption, $\H_*(\K;\k)$ has an additive basis of the form $\{\sum_{I\in\MF(\K)}\mu_{I,q}'[\partial\Delta_I]:q\in Q\}$, where $\mu_{I,q}'\in\k$. There are four cases:

\textit{Case 1:} if $\k=\FF_p$ is a prime field, then we can take integral lifts of $\mu_{I,q}'\in\FF_p$ as $\mu_{I,q}$.

\textit{Case 2:} if $\k$ is any principal ideal domain of positive characteristic, then $\chr\k=p>0$ is a prime number. In particular, $\k$ is a $\FF_p$-algebra, and so $\H_*(\K;\k)\cong\H_*(\K;\FF_p)\otimes_{\FF_p}\k$, and the result follows from the case $\k=\FF_p$. 

\textit{Case 3:} if $\k$ is a subring of $\QQ$, then $\mu_{I,q}'$ are rational numbers with $\k$-invertible denominators. Multiplying by the common denominators, we obtain a basis with integral coefficients, and the result follows.

\textit{Case 4:} if $\k$ is any principal ideal domain of characteristic zero, define the subring $R$ of $\mathbb{Q}$ by $R:=\ZZ[1/p:~p\in\ZZ\text{ is prime, }p\in\k^\times]\subseteq\QQ$. Then $\k$ is a $R$-algebra, and $p\in\k^\times$ if and only if $p\in R^\times$. By Proposition \ref{prp:HMF-criterion} applied to $\k$ and $R$, we have: $\H_*(\K;\k)\cong \H_*(\K;\ZZ)\otimes\k$ is a free $\k$-module, and $\H_*(\K;R)\cong \H_*(\K;\ZZ)\otimes R$ is generated by missing faces.

We now check that $\H_*(\K;\ZZ)\otimes R$ is a free $R$-module. Write $\H_*(\K;\ZZ)$ as a direct sum of copies of $\ZZ$ and $\ZZ/p^i$ for $p$ prime. If $\H_*(\K;\ZZ)\otimes R$ is not free over $R$, then there is a direct summand $\ZZ/p^i$ with $p\notin R^\times$. Then $p\notin\k^\times$ by definition of $R$, so the $\k$-module $\H_*(\K;\ZZ)\otimes\k$ has a torsion summand $\ZZ/p^i\otimes\k$. However, $\H_*(\K;\ZZ)\otimes\k$ is free over $\k$, which is a contradiction.

It follows that $\H_*(\K;R)$ is a free $R$-module generated by missing faces, so by Case 3 it has a basis of the required form. Finally, $\H_*(\K;\k)\cong\H_*(\K;\ZZ)\otimes\k\cong (\H_*(\K;\ZZ)\otimes R)\otimes_R\k\cong\H_*(\K;R)\otimes_R\k$, so the same set of elements is a basis for $\H_*(\K;\k)$.
\end{proof}

\subsection{HMF complexes and homology fillable complexes}
\label{subsec:HMF and THF}
\emph{Totally homology fillable} simplicial complexes were defined by Iriye and Kishimoto in \cite{fat-wedge}. For this class, the homotopy theory of simplicial complexes simplifies significantly, see \cite[Theorem 7.8, Corollary 7.9, Proposition 7.11, Theorem 8.21]{fat-wedge}. In particular, $\ZK$ is a wedge of spheres.

\begin{dfn}[{\cite[Definition 7.6]{fat-wedge}}]
A simplicial complex $\K$ is \emph{homology fillable} if
each connected component $\L$ of $\K$ satisfies the following conditions:
\begin{itemize}
    \item For any prime $p$, there exists a subset $A^{(p)}\subseteq\MF(\L)$ such that $\H_*(\L\cup A^{(p)};\ZZ/p)=0$;
    \item $\L\cup\MF_{\geq 3}(\L)$ is simply connected.
\end{itemize}
The complex $\K$ is \emph{totally homology fillable} if $\K_J$ is homology fillable for all $\varnothing\neq J\subseteq[m]$.
\end{dfn}

We will relate this condition to the HMF property.

\begin{rmk}
In \cite[Definition 7.6]{fat-wedge}, instead of the condition $\pi_1(\L\cup\MF_{\geq 3}(\L))=0$ a stronger condition $\pi_1(\L\cup\MF(\L))$ was required. However, \cite[Theorem 7.8]{fat-wedge} still holds for the corrected definition, with the same proof. We only need to verify that $\L\cup A^{(p)}\subseteq \L\cup\MF_{\geq 3}(\L)$. Indeed, since $\H_0(\L;\k)=0$, by Lemma \ref{lmm:MV-corollary} below we obtain an exact sequence $\H_1(\L\cup A^{(p)};\ZZ/p)\to (\ZZ/p)^n\to 0$, $n=|A^{(p)}\cap \MF_2(\L)|$. It follows that $A^{(p)}\cap \MF_2(\L)=\varnothing$, i.e. $A^{(p)}\subset\MF_{\geq 3}(\L).$

We also note that the proof of \cite[Proposition 8.17]{fat-wedge} shows that $\pi_1(\L\cup \MF_{\geq 3}(\L))=0$ instead of $\pi_1(\L\cup\MF(\L))=0$, so \cite[Proposition 8.20]{fat-wedge} shows that dual-SCM complexes satisfy the corrected definition of homology fillable complexes. 
\end{rmk}

\begin{lmm}
\label{lmm:MV-corollary}
Let $\K$ be a simplicial complex. For $A\subseteq\MF(\K)$, there is a long exact sequence
$$\dots\to\H_{i+1}(\K;\k)\to\H_{i+1}(\K\cup A;\k)\to\bigoplus_{I\in A_{i+2}}\k\cdot x_I\overset{\mu}\longrightarrow \H_i(\K;\k)\to\H_i(\K\cup A;\k)\to\dots,$$
where $A_i:=\{I\in A:|I|=i\}$ and
$\mu(x_I)=[\partial\Delta_I]\in\H_i(\K;\k)$.
\end{lmm}
\begin{proof}
This is a long exact sequence of the pair $(|\K\cup A|,|\K|)$ in reduced homology; note that $|\K\cup A|/|\K|\cong \bigvee_{I\in A}|\Delta_I|/|\partial\Delta_I|\cong\bigvee_{I\in A}S^{|I|-1}$.
\end{proof}

\begin{prp}
\label{prp:hf-and-shmf}
For a field $\k$, the following conditions are equivalent:
\begin{enumerate}
    \item $\K$ is HMF over $\k$;
    \item For each $\varnothing\neq J\subseteq[m]$, there exists a subset $A_J\subseteq\MF(\K_J)$ such that $\{[\partial\Delta_I]:I\in A_J\}$ is a basis of $\H_*(\K_J;\k)$;
    \item For each $\varnothing\neq J\subseteq[m]$, there exists a subset $A_J\subseteq \MF(\K_J)$ such that $\H_*(\K_J\cup A_J;\k)=0$.
\end{enumerate}
\end{prp}
\begin{proof}
Clearly, (2) implies (1) for any ring. Since $\k$ is a field, each spanning set contains a basis, so (1) implies (2). Finally (2) and (3) are equivalent by Lemma \ref{lmm:MV-corollary}.
\end{proof}
A graph is \emph{chordal} if for each cycle of length $\geq 4$, there is an edge between non-adjacent vertices.
\begin{lmm}
\label{lmm:chordal pi1} The following results hold:
\begin{enumerate}
    \item If $\K$ is HMF over $\k$, then its 1-skeleton is a chordal graph.
    \item If $\K$ is a connected simplicial complex with chordal 1-skeleton, then $\pi_1(\K\cup\MF_{\geq 3}(\K))=0$.
\end{enumerate}
\end{lmm}
\begin{proof}
If $\K^1$ has a chordless cycle on the vertex set $J$, then $\K_J$ is a boundary of an $|J|$-gon. Since $\H_*(\K_J;\k)$ is generated by missing faces, $|J|\leq 3$. This proves (1).

For (2), note that $\pi_1(|\K|)$ is generated by classes of simple cycles. Since $\K^1$ is chordal, it is generated by classes of missing triangles, so $\pi_1(\K)\to \pi_1(\K\cup\MF_{\geq 3}(\K))$ is trivial. On the other hand, this map is surjective since it attaches cells of dimension $\geq 2$.
\end{proof}
\begin{prp}
\label{prp:thf-is-hmf}
The following conditions are equivalent:
\begin{enumerate}
    \item $\K$ is totally homology fillable;
    \item For any prime $p$, $\K$ is HMF over $\ZZ/p$;
    \item $\K$ is HMF over $\ZZ$, and $H_*(\K_J;\ZZ)$ is torsion-free for all $J\subseteq[m]$.
\end{enumerate}
\end{prp}
\begin{proof}
Conditions (2) and (3) are equivalent by Proposition \ref{prp:HMF-criterion}.

By Proposition \ref{prp:hf-and-shmf}, $\K$ satisfies (1) if and only it satisfies (2) and the following condition: for each connected full subcomplex $\K_J$ of $\K$, we have $\pi_1(|\K_J\cup \MF(\K_J)|)=0$. Lemma \ref{lmm:chordal pi1} implies that this condition follows from (2), so (1) and (2) are equivalent.
\end{proof}
\begin{rmk}
A complex $\K$ is \emph{dual-SCM} over $\k$ if its Alexander dual $\K^\vee$ is a sequentially Cohen--Macaulay complex. By \cite[Theorem 8.21]{fat-wedge}, each dual-SCM complex is totally homology fillable, hence, is HMF over each $\ZZ/p$. In \cite[Proof of Corollary 4.7]{amelotte-briggs}, this result is interpreted in terms of commutative algebra: over a field $\k$,
\begin{itemize}
    \item $\K$ is HMF if and only if the Stanley--Reisner ideal $I_\K=\Ker(\k[v_1,\dots,v_m]\to\k[\K])$ is a quasi-Koszul module over $\k[v_1,\dots,v_m]$ \cite[Proposition 2.4]{amelotte-briggs};
    \item $\K$ is dual-SCM if and only if $I_\K$ is a Koszul module (by results of Herzog, Hibi and R\"omer).
\end{itemize}
It would be interesting to know if the HMF-presented property has a description in terms of $I_\K$.
\end{rmk}
\subsection{General results on loop homology}

Let $\k$ be a commutative ring with unit.
By \cite[Proposition 8.4.10]{BP15}, there is an isomorphism of algebras $$H_*(\Omega\djk;\k)\cong H(\Omega_*\kKc)\cong \Ext_{\k[\K]}(\k,\k),\quad H_n(\Omega\djk;\k)\cong\bigoplus_{n=-i+2j}\Ext^i_{\k[\K]}(\k,\k)_{2j},$$ where $\Omega_*\kKc$ is the cobar construction for the Stanley--Reisner coalgebra $\kKc:=(\k[\K])^*$. By a result of Franz \cite[Proposition 6.5]{franz-hga}, it is an isomorphism of Hopf algebras when $H_*(\Omega \djk;\k)$ is a free $\k$-module. For $\alpha=\sum_{i=1}^m \alpha_ie_i\in\Zm$, denote $|\alpha|:=\sum_i \alpha_i$, $\supp\alpha:=\{i\in[m]:~\alpha_i\neq 0\}$.
The cobar construction is the following dg-algebra:
$$\Omega_*\kKc=(T(\chi_\alpha:\alpha\in\Zm\setminus\{0\},~\supp\alpha\in\K),d),\quad d(\chi_\alpha):=\sum_{\alpha=\beta+\gamma}\chi_\beta\chi_\gamma,\quad \deg\chi_\alpha:=2|\alpha|-1.$$ 

For each $i\in[m]$, let $u_i\in H_1(\Omega\djk;\k)$ be the element represented by the cycle $\chi_i\in(\Omega_*\kKc)_{1}$. For each $I\in\MF_{\geq 3}(\K)$, let $w_I\in H_{2|I|-2}(\Omega\djk;\k)$ be the element represented by the cycle
$$\psi_I:=\sum_{I=A\sqcup B}\chi_A\chi_B\in\Omega_*\kKc.$$ Note that there is an embedding of dg-coalgebras $\Omega_*\kKc\subseteq\Omega_*\k\langle\Delta_{[m]}\rangle$. Since $\chi_I\in\Omega_*\k\langle \Delta_{[m]}\rangle$ and $d(\chi_I)=\psi_I\in\Omega_*\k\langle\Delta_{[m]}\rangle$, it follows that $d(\psi_I)=d^2(\chi_I)=0\in\Omega_*\k\langle\Delta_{[m]}\rangle$, so $\psi_I\in\Omega_*\kKc$ is a cycle.

\begin{rmk}
The elements $\psi_I$ and $w_I\in H_*(\Omega\djk;\k)$ are defined for any $I\subseteq[m]$ such that $\partial\Delta_I\subseteq\K$, i.e. for any $I\in\K\sqcup\MF(\K)$. However,
\begin{itemize}
    \item If $I\in\K$, then $\psi_I=d(\chi_I)$ is a boundary in $\Omega_*\kKc$, so $w_I=0$;
    \item If $I\in\MF_2(\K)$, $I=\{i,j\}$, we have $\psi_I=\chi_i\chi_j+\chi_j\chi_i$, so $w_I=u_iu_j+u_ju_i=[u_i,u_j]$ is a decomposable element.
\end{itemize}
Hence, we usually only consider the remaining case $I\in\MF_{\geq 3}(\K)$.
\end{rmk}
\begin{lmm}
\label{lmm:standard-relations-in-hodj}
For any $\K$, the classes $u_i$ and $w_I$ of $H_*(\Omega\djk;\k)$ satisfy the following relations:
\begin{enumerate}
    \item $u_i^2=0$, for all $i=1,\dots,m$; $[u_i,u_j]=0$ for all $\{i,j\}\in\K$.
    \item $[u_i,w_I]=0$ if $i\in I\in\MF(\K)$.
    \item $\sum_{\begin{smallmatrix}
        i\in L:\\L\setminus i\notin\K
    \end{smallmatrix}}[u_i,w_{L\setminus i}]=0$ for all $L\in\AMF(\K)$.
\end{enumerate}
\end{lmm}
\begin{proof}
By definition, $\chi_i^2=d(\chi_{2e_i})$ for all $i=1,\dots,m$, and $[\chi_i,\chi_j]=\chi_i\chi_j+\chi_j\chi_i=d(\chi_{\{i,j\}})$ for $\{i,j\}\in\K$. Hence, (1) holds.

For $I\in\MF(\K)$ and $i \in I$, consider the element $\tau:=[\chi_I,\chi_i]-d(\chi_{I+e_i})\in\Omega_*\k\langle\Delta_{[m]}\rangle$. It follows from the definition that $\tau=-\sum_{A\sqcup B}(\chi_{A+e_i}\chi_B+\chi_A\chi_{B+e_i})$, and so $\tau\in\Omega_*\kKc$. On the other hand, we have $d(\tau)=d([\chi_I,\chi_i])=[\psi_I,\chi_i]+[\chi_I,0]$. It follows that $[\psi_I,\chi_i]$ is a boundary in $\Omega_*\kKc$, so the relation (2) follows.

Finally, for $L\in\AMF(\K)$, consider the element
$$\theta:=d(\chi_L)-\sum_{i\in L}[\chi_i,\chi_{L\setminus i}]=\sum_{\begin{smallmatrix}I=A\sqcup B:\\|A|,|B|\geq 2\end{smallmatrix}}\chi_A\chi_B\in\Omega_*\kKc.$$
Since $d(\chi_i)=0$, we obtain
$$d(\theta)=\sum_{i\in L}[\chi_i,d(\chi_{L\setminus i})]=\sum_{\begin{smallmatrix}i\in L:\\L\setminus i\notin\K\end{smallmatrix}}[\chi_i,\psi_{L\setminus i}]+\sum_{\begin{smallmatrix}
i\in L:\\L\setminus i\in\K
\end{smallmatrix}}
d([\chi_i,\chi_{L\setminus i}])
$$
It follows that the first summand is a boundary in $\Omega_*\kKc$, so the relation (3) follows.
\end{proof}
\begin{rmk}
If $I\in\MF_2(\K)$, we have $w_{\{i,j\}}=[u_i,u_j],$ so the relation (2) becomes the tautological identity $[u_i,[u_i,u_j]]=0$. Similarly, for $L\in\AMF_3(\K)$ the relation (3) follows from (1) and the Jacobi identity.

In general, this is not an exhaustive list of relations. For example, $[w_I,w_J]=0$ if $\partial\Delta_I\ast\partial\Delta_J\subseteq\K$.
\end{rmk}

For $J=\{i_1<\dots<i_s\}\subseteq [m]$ and $x\in H_*(\Omega\djk;\k)$, we denote
\begin{equation}
\label{eqn:u-and-c-definition}
\widehat{u}_J=u_{i_1}u_{i_2}\dots u_{i_s},\quad c(J,x):=[u_{i_1},[u_{i_2},\dots [u_{i_s},x]\dots]];
\end{equation}
in particular, $\widehat{u}_\varnothing=1$ and $c(\varnothing,x)=x$. Define the Koszul sign as
$$\theta(I,J)=\#\{(i,j)\in I\times J:~i>j\},\quad \theta(i,J)=|J_{<i}|.$$

We first study the case where $\K$ is a $1$-neighbourly complex. In Section \ref{sec:Presentations}, we apply Theorem \ref{thm:uxk_colimit_decomposition_intro} to obtain results for $\K$ in the general case.

\begin{lmm}
If $\K$ is a $1$-neighbourly complex, then
\begin{equation}
\label{eqn:ui_commute_corollary}
\widehat{u}_I\cdot\widehat{u}_J=\begin{cases}(-1)^{\theta(I,J)} u_{I\sqcup J},& I\cap J=\varnothing,\\ 0,&I\cap J\neq \varnothing;\end{cases}\quad c(I,c(J,x))=\begin{cases}(-1)^{\theta(I,J)} c(I\sqcup J,x),&I\cap J=\varnothing,\\0,&I\cap J\neq\varnothing.\end{cases}
\end{equation}
\end{lmm}
\begin{proof}
By part (1) of Lemma \ref{lmm:standard-relations-in-hodj}, the elements $u_1,\dots,u_m$ commute, and $[u_i,[u_i,y]]=0$.
\end{proof}

Finally, we will require a geometric interpretation of the elements $u_i\in H_1(\Omega\djk;\k)$, and $w_J\in H_{2|J|-2}(\Omega\djk;\k)$. Let $\ZK := (D^2,S^1)^\K$ be the \textit{moment-angle complex} associated to $\K$. Information about the homology groups and homotopy groups of $\ZK$ will be important. For homology, by \cite[Corollary 2.10]{bbcg}, there is an isomorphism 
\begin{equation}
\label{eqn:hochster-formula}
\H_*(\mathcal{Z}_{\K};\k) \cong \bigoplus_{I \notin \K} \H_*(\Sigma^{1+|I|} |\K_I|;\k)\cong\bigoplus_{\varnothing\neq J\subseteq[m]}\H_{*-|J|-1}(\K_J;\k).
\end{equation}

For the homotopy groups, by \cite[Theorem 4.3.2]{BP15}, the inclusion of $\djk$ into $(\mathbb{C}P^\infty)^m$ gives a homotopy fibration $\ZK \to \djk \to (\mathbb{C}P^\infty)^m$. In particular, the long exact sequence of homotopy groups associated to this homotopy fibration implies that $\pi_k(\ZK) \cong \pi_k(\djk)$ for all $k \geq 3$. Moreover, by \cite[(8.16)]{BP15} there is a homotopy splitting $\Omega\djk\simeq\Omega\ZK\times\Omega(\CC P^\infty)^m$, hence the induced map $H_*(\Omega\ZK;\k)\to H_*(\Omega\djk;\k)$ is injective.

By definition of $\ZK$, if $I \in MF(\K)$, there is a homeomorphism $\mathcal{Z}_{\K_I}=\Z_{\partial\Delta_I} \cong S^{2|I|-1}$. In particular, there is a canonical map $S^{2|I|-1} \to \ZK \to \djk$ induced by the simplicial inclusion $\K_I \to \K$. With this setup, we can state the geometric interpretation of $u_i$ and $w_I$.

\begin{prp}
\label{prp:c(J,wI)-geometric-interpretation}
Let $t_1,\dots,t_m\in\pi_2(\djk)$ be the classes represented by the standard inclusions $S^2\hookrightarrow \CC P^\infty\subseteq\djk$. For $I\in\MF(\K)$, let $\omega_I\in\pi_{2|I|-1}(\djk)$ be the class represented by the map $S^{2|I|-1}\cong\Z_{\K_I}\to\djk$. Then
\begin{enumerate}
    \item $u_i\in H_1(\Omega\djk;\k)$ is the Hurewicz image of the adjoint map $\overline{t}_i\in\pi_1(\Omega\djk)$, and the element $w_I\in H_{2|I|-2}(\Omega\djk;\k)$ is the Hurewicz image of the adjoint map $\overline{\omega}_I\in\pi_{2|I|-2}(\Omega\djk)$.
    \item For any $J\subseteq[m]$ and $I\in\MF(\K_J)$, the element $c(J\setminus I,w_I)\in H_{|J|-2}(\djk;\k)$ is in the subalgebra $H_*(\Omega\ZK;\k)\subseteq H_*(\Omega\djk;\k)$. 
    \item For the homology suspension map $\sigma:\H_*(\Omega\ZK;\k)\to\H_{*-1}(\ZK;\k)\cong\bigoplus_{J\subseteq[m]}\H_{*-|J|}(\K_J;\k)$, we have $\sigma(c(J\setminus I,w_I))=(-1)^{|J\setminus I|+\theta(I,J)}[\partial\Delta_I]\in\H_*(\K_J;\k)$.
\end{enumerate}
\end{prp}
\begin{proof}
(1) follows from \cite[Corollary 5.2 and Proposition 5.3]{zhuravleva}. By (1) and the fact that Samelson products map to commutators in homology, $c=c(J,w_I)=[u_{i_1},\dots[u_{i_k},w_I]\dots]$ is the Hurewicz image of the element adjoint to the iterated Whitehead bracket $[t_{i_1},\dots[t_{i_k},\omega_I]\dots]$. Since $\pi_{\geq 3}(\ZK)\to\pi_{\geq 3}(\djk)$ is an isomorphism, $c$ is in the image of $H_*(\Omega\ZK;\k)\to H_*(\Omega\djk;\k)$, so (2) follows.

Identity (3) follows from \cite[Lemma 4.1]{abramyan} (see also \cite[Corollary 4.13]{amelotte-briggs}). In more detail, let $I=\{i_1,\dots,i_p\}$ and $A=J\setminus I=\{j_0<\dots<j_{n-1}\}$. In the notation of \cite{abramyan}, the element $\sigma(c(A,w_I))\in\H_*(\ZK)$ is the Hurewicz image of the iterated higher Whitehead product
$[\mu_{j_0},[\mu_{j_1},\dots[\mu_{j_{n-1}},[\mu_{i_1},\dots,\mu_{i_p}]]\dots]]\in\pi_*(\ZK)$. By \cite[Lemma 4.1]{abramyan}, it is represented by the cellular chain
$$S_{j_0}S_{j_1}\dots S_{j_{n-1}}\sum_{t=0}^p D_{i_1}\dots D_{i_{t-1}}S_{i_t}D_{i_{t+1}}\dots D_{i_p}=\sum_{i\in I}(-1)^{|A_{>i}|}\chi(A\sqcup i,I\setminus i)\in\mathcal{C}_*(\ZK;\k),$$
which, by \cite[Theorem 3.3]{abramyan}, corresponds to the class
\[\Big[\sum_{i\in I}(-1)^{|A_{>i}|+\sum_{\ell\in I\setminus i}|J_{<\ell}|} e_{I\setminus\{i\}}\Big]=(-1)^{|A|+\theta(I,J)}[c_I]=(-1)^{|J\setminus I|+\theta(I,J)}[\partial \Delta_I]\in\H_*(\K_J;\k).\qedhere\]

\end{proof}
\begin{crl}
\label{crl:w_I-primitive}
Suppose that $H_*(\Omega\djk;\k)$ is a free $\k$-module. Then for all $1 \leq i \leq m$ and $I \in \MF(\K)$, the elements $u_i,w_I\in H_*(\Omega\djk;\k)$ are primitive.
\end{crl}
\begin{proof}
By Part (1) of Proposition \ref{prp:c(J,wI)-geometric-interpretation}, these elements are in the Hurewicz image, so they are primitive since primitive elements map to primitive elements.
\end{proof}

\begin{rmk}
In the notation of \cite[\S 2]{abramyan-panov}, the classes $\omega_I$ are ``canonical higher Whitehead brackets'' of $t_i$, $i\in I$.
More generally, let $w\in\pi_*(\djk)$ be a canonical \emph{iterated} higher Whitehead bracket of $t_1,\dots,t_m$. Then \cite[Lemma 3.2]{abramyan-panov} identifies the Hurewicz image of $\alpha$ in $\H_*(\ZK)$. Under the isomorphism \eqref{eqn:hochster-formula}, this Hurewicz image is of the form $\pm[\partial\Delta_{I_1}]\ast\dots\ast [\partial\Delta_{I_s}]\in\H_*(\K_J;\k)$.
\end{rmk}

\subsection{Loop homology for HMF-presented complexes}

In the following, we will use the notation in \eqref{eqn:u-and-c-definition}.
\begin{prp}
\label{prp:HMF-without-torsion-additive-basis}
Let $\K$ be a simplicial complex and $\k$ be a principal ideal domain. Assume that $\K$ is HMF over $\k$, and $H_*(\ZK;\k)$ is a free $\k$-module. Then 
\begin{enumerate}
    \item For each $J\subseteq[m]$, $J\neq\varnothing$, the $\k$-module $\H_*(\K_J;\k)$ has an additive basis of the form
    $$B_0(J)=\{\sum_{I\in\MF(\K_J)}\mu_{I,q} [\partial\Delta_I],\quad q\in Q_J\}$$
    for some $\mu_{I,q}\in\ZZ$ and finite sets $Q_J$.
    \item $H_*(\Omega\ZK;\k)$ is a free associative $\k$-algebra on the set of generators
    $$B=\Big\{\sum_{I\in\MF(\K_J)}\mu_{I,q}c(J\setminus I,w_I):\quad J\subseteq[m],~q\in Q_J\Big\}.$$
    \item $H_*(\Omega\djk;\k)$ has an additive basis
    $B'=\{b_1b_2\dots b_P\cdot\widehat{u}_J:\quad b_1,\dots,b_P\in B,~J\subseteq[m]\}.$
    \item $H_*(\Omega\djk;\k)$ is multiplicatively generated by $\{u_1,\dots,u_m\}\sqcup\{w_I:I\in\MF_{\geq 3}(\K)\}$.
\end{enumerate}
In particular, $H_*(\Omega\ZK;\k)$ is a free associative $\k$-algebra, and $H_*(\Omega\djk;\k)$ is a free $\k$-module.
\end{prp}
\begin{proof}
By \eqref{eqn:hochster-formula}, $\H_*(\K_J;\k)$ is a free $\k$-module for each $J$, so
(1) follows from Proposition \ref{prp:integer_coefficient_basis}.

By Proposition \ref{prp:c(J,wI)-geometric-interpretation}, the elements $\sum \mu_{I,q}c(J,\setminus I,w_I)\in B$ are Hurewicz images of classes in $\pi_*(\Omega\ZK)$ adjoint to linear combinations of iterated Whitehead products. Therefore, we obtain a map $f:W\to\ZK$ from a simply connected wedge of spheres which induces an isomorphism $H_*(W;\k)\cong H_*(\ZK;\k)$. By naturality of the Serre spectral sequence, $f$ induces an isomorphism $H_*(\Omega W;\k)\cong H_*(\Omega\ZK;\k)$, so $H_*(\Omega\ZK;\k)$ is a free associative algebra generated by Hurewicz images of elements adjoint to the spherical classes. This proves the statement (2).

By \cite[Proposition 3.7]{vylegzhanin25}, $H_*(\Omega\djk;\k)\cong H_*(\Omega\ZK;\k)\otimes\Lambda[u_1,\dots,u_m]$ as $H_*(\Omega\ZK;\k)$-modules, so
(3) follows from (2). Finally, (4) is immediate from (3).
\end{proof}
\begin{rmk}
It follows from the proof that $W\to \ZK$ is a homotopy equivalence after localisation away from $\{p:p\notin\k^\times\}$. This is a result of Amelotte and Briggs, see \cite[Corollary 4.6]{amelotte-briggs}. 
\end{rmk}

\begin{lmm}
\label{lmm:HMF-presented-map-gJ}
Let $\K$ be a $1$-neighbourly simplicial complex which is HMF-presented over a ring $\k$. Then, for each $J\subseteq[m]$, there is a well defined map of $\k$-modules
$$g_J:\H_*(\K_J;\k)\to H_*(\Omega\djk;\k),\quad [\partial\Delta_I]\mapsto (-1)^{\theta(I,J)+|J\setminus I|}c(J\setminus I,w_I).$$
\end{lmm}
\begin{proof}
Since $\K$ is HMF-presented, the $\k$-module $\H_*(\K_J;\k)$ is additively generated by the classes $\{[\partial\Delta_I]:I\in\MF(\K_J)\}$ modulo the relations \eqref{eqn:amf_homology_relation} for all $L\in\AMF(\K_J)$. It remains to check that images $g_J([\partial\Delta_I])$ satisfy these relations. Let $n:=|L|-1$ and $A:=J\setminus L$. Under the map $g_J$, the summand $(-1)^{\theta(L\setminus i,i)}[\partial\Delta_{L\setminus i}]$ is mapped to
$$(-1)^{\theta(L\setminus i,i)+\theta(L\setminus i,J)+|A|+1}c(A\sqcup i,w_{L\setminus i})=(-1)^{\theta(L\setminus i,i)+\theta(L\setminus i,J)+|A|+1+\theta(A,i)}c(A,[u_i,w_{L\setminus i}]).$$
The sign is equal to
$$\theta(L,i)+\theta(L\setminus i,J)+1+|A|+\theta(A,i)=\theta(A,i)+\theta(L,i)+\theta(i,J)+\theta(L,J)+1+|A|=|J|+\theta(L,J)+1.$$
Therefore $$g_J(\rho_L)=(-1)^{|J|+\theta(L,J)+1}\sum_{i\in L}[u_i,w_{L\setminus i}],$$
which is zero by part (3) of Lemma \ref{lmm:standard-relations-in-hodj}.
\end{proof}

\begin{thm}
\label{thm:hodj-presentation-for-torsionfree-hmf-presented}
Suppose that $\K$ is a $1$-neighbourly simplicial complex which is HMF-presented over a principal ideal domain $\k$. Assume that $H_*(\ZK;\k)$ is a free $\k$-module. Then the Hopf algebra $H_*(\Omega\djk;\k)$ is isomorphic to
$$A=T(u_1,\dots,u_m;~w_I,~I\in\MF(\K))/(R),$$
where $u_i$, $w_I$ are primitive, and $R$ is generated by relations of three types:
\begin{enumerate}
    \item $u_i^2=0$, $[u_i,u_j]=0$ for all $i=1,\dots,m$;
    \item $[u_i,w_I]=0$ if $i\in I$;
    \item $\sum_{\begin{smallmatrix}
        i\in L:\\L\setminus i\notin\K
    \end{smallmatrix}}[u_i,w_{L\setminus i}]=0,$ $L\in\AMF(\K)$.
\end{enumerate}
\end{thm}
\begin{proof}
By part (3) of Proposition \ref{prp:HMF-without-torsion-additive-basis}, $H_*(\Omega\djk;\k)$ is a Hopf algebra since it is a free $\k$-module. The elements $u_i$, $w_I$ are primitive by Corollary \ref{crl:w_I-primitive}.

By Lemma \ref{lmm:standard-relations-in-hodj}, the relations (1)-(3) hold in $H_*(\Omega\djk;\k)$. Thus there is a well defined map of algebras $f:A\to H_*(\Omega\djk;\k)$. By part (3) of Proposition \ref{prp:HMF-without-torsion-additive-basis}, it is surjective.

To prove that $f$ is an isomorphism, it is sufficient to prove that the $\k$-module $A$ is spanned by the set $B'$ defined in part (3) of Proposition \ref{prp:HMF-without-torsion-additive-basis}. If $A$ is spanned by $B'$, then there is a surjective map of $\k$-modules $h:H_*(\Omega\djk;\k)\to A$. The composite $f\circ h$ is a surjective endomorphism of $H_*(\Omega\djk;\k)$, which is a free $\k$-module of finite type. Since $\k$ is a principal ideal domain, it follows that $f\circ h$ is an isomorphism, so $f$ is injective.

By definition, $A$ is spanned by words on $u_j$'s and $w_I$'s. Let $x$ be such a word. Whenever there is a subword of the form $u_jc(J,w_I)$, we replace it with $\pm c(J,w_I)u_j\pm c(\{j\},c(J,w_I))$ by the Jacobi identity. Using the relations \eqref{eqn:ui_commute_corollary}, we thus express $x\in A$ as a linear combinations of words of the form $c_1c_2\dots c_N\cdot \widehat{u}_K$, where each $c_i$ is of the form $c(A,w_I)$, $A\subseteq[m]$, $I\in\MF(\K)$.

If $i\in A\cap I$, the relations $[u_i,w_I]=0$ and \eqref{eqn:ui_commute_corollary} imply that $c(A,w_I)=\pm c(A\setminus i,[u_i,w_I])=0$. Thus we can assume that each $c_i$ is of the form $c_i=c(J\setminus I,w_I)$, $I\in\MF(\K_J)$.  It follows that $c_i=g_J([\partial\Delta_I])$, where the map $g_J:\H_*(\K_J;\k)\to H_*(\Omega\djk;\k)$ was defined in Lemma \ref{lmm:HMF-presented-map-gJ}. Since $[\partial\Delta_I]\in\spn_\k(B_0(J))$, $c_i\in\spn_\k(B)$ by definition of $B$. Hence $x$ is a linear combination of words of the form $c_1'c_2'\dots c_N'\cdot\widehat{u}_K$, $c_i'\in B$. This shows that $x\in\spn_\k(B')$, which completes the proof.
\end{proof}

In \cite{S}, it was shown that if $\K$ is the $d$-skeleton of a simplex, then $\Omega \djk$ decomposes as a product of spheres and loops on spheres. This decomposition only gives information about $H_*(\Omega \djk;\k)$ as a module, namely it is a free $\k$-module for any commutative ring with unit $\k$. We can use Theorem \ref{thm:hodj-presentation-for-torsionfree-hmf-presented} to give a presentation of $H_*(\Omega \djk;\k)$ as a Hopf algebra. Special cases of this result were obtained by Panov and Ray, see \cite[Example 10.2]{panov-ray}.

\begin{crl}
\label{crl:dj-for-skeleta-presentation}
Suppose that $\K=\sk_d\Delta_{[m]}$, $1\leq d\leq m-2$. Then, for any principal ideal domain $\k$, $H_*(\Omega\djk;\k)$ is isomorphic to the Hopf algebra
$$A=T(u_1,\dots,u_m;~w_I,~I\subseteq[m],~|I|=d+2)/(R),$$
where $u_i$, $w_I$ are primitive, and $R$ is generated by relations of three types:
\begin{enumerate}
    \item $u_i^2=0$, $i=1,\dots,m$; $[u_i,u_j]=0$ for all $i$;
    \item $[u_i,w_I]=0$ if $i\in I$;
    \item $\sum_{i\in J}[u_i,w_{J\setminus i}]=0$ for all $J\subseteq[m]$, $|J|=d+3$.
\end{enumerate}
\end{crl}
\begin{proof}
$\K$ is HMF-presented by Lemma \ref{lmm:skeleta-HMF-presented}, so the result follows from Theorem \ref{thm:hodj-presentation-for-torsionfree-hmf-presented}.

\end{proof}
\section{Applications to Polyhedral Products}
\label{sec:Polyhedralproducts}

In this section, we use Theorem \ref{thm:general-theorem-for-pushout} to give a colimit decomposition of the loop homology of weighted polyhedral products when $H_*(\Omega \ux^{\K,c};\k)$ is a free $\k$-module. We then apply this to give presentations of the loop homology of special classes of polyhedral products.

\subsection*{A colimit decomposition}

In this section, we prove a general colimit decomposition of weighted polyhedral products, from which Theorem \ref{thm:uxk_colimit_decomposition_intro} follows. 

\begin{thm}
\label{thm:uxk_colimit_decomposition}
    Let $\K$ be a simplicial complex, and $(X_1,\rho^1),\cdots,(X_m,\rho^m)$ be a sequence of simply connected spaces with associated power maps. Let $c$ be a power sequence. Let $\k$ be a principal ideal domain such that $H_*(\Omega \ux^{\K,c};\k)$ is a free $\k$-module. Then for each $I \in \Kf$, the maps $H_*(\Omega \ux^{\K_I,c}) \to H_*(\Omega \ux^{\K,c})$ induced by the simplicial inclusion define an isomorphism of Hopf algebras \[H_*(\Omega \ux^{\K,c};\k) \cong \colim_{I \in \Kf} H_*(\Omega \ux^{\K_I,c};\k).\]  Moreover, this equivalence is natural with respect to simplicial inclusions. 
\end{thm}
\begin{proof}
    If $\K$ is a simplicial complex such that $\K^1$ is complete, then the corresponding diagram has a terminal object and the result is clear. 

    Now suppose $\K$ is a simplicial complex such that $\K^1$ is not complete. By \cite[Lemma 4.4]{S}, there is a pushout of simplicial complexes \[\begin{tikzcd}
	{\K_{A\cap B}} & {\K_A} \\
	{\K_B} & \K
	\arrow[from=1-1, to=1-2]
	\arrow[from=1-1, to=2-1]
	\arrow[from=1-2, to=2-2]
	\arrow[from=2-1, to=2-2]
\end{tikzcd}\] for some $A,B\subsetneq [m]$, induced by inclusions of full subcomplexes. By Lemma ~\ref{lem:pushout}, this induces a (homotopy) pushout of polyhedral products \[\begin{tikzcd}
	{\ux^{\K_{A\cap B},c}} & {\ux^{\K_B,c}} \\
	{\ux^{\K_A,c}} & {\ux^{\K,c}.}
	\arrow[from=1-1, to=1-2]
	\arrow[from=1-1, to=2-1]
	\arrow[from=1-2, to=2-2]
	\arrow[from=2-1, to=2-2]
\end{tikzcd}\] 

Lemma ~\ref{lem:retraction} implies that each map in the homotopy pushout has a left homotopy inverse. In particular, since $\k$ is a principal ideal domain and $H_*(\Omega \ux^\K;\k)$ is a free $\k$-module, the loop homology of each object in the pushout is a free $\k$-module. 
Therefore, ommiting the coefficients of homology, by part (2) of Theorem \ref{thm:general-theorem-for-pushout}, there is a pushout of Hopf algebras \[\begin{tikzcd}
	{H_*(\Omega\ux^{\K_{A\cap B},c})} & {H_*(\ux^{\K_A,c})} \\
	{H_*(\ux^{\K_B,c})} & {H_*(\ux^{\K,c}).}
	\arrow[from=1-1, to=1-2]
	\arrow[from=1-1, to=2-1]
	\arrow[from=1-2, to=2-2]
	\arrow[from=2-1, to=2-2]
\end{tikzcd}\] By the inductive hypothesis, up to natural isomorphism, this pushout can be written as the pushout \begin{equation}\label{eqn:kanext}\begin{tikzcd}
	{\colim_{I \in \K_{A\cap B}^f}H_*(\Omega\ux^{\K_I,c})} & {\colim_{I \in \K_A^f}H_*(\Omega\ux^{\K_I,c})} \\
	{\colim_{I \in \K_B^f}H_*(\Omega\ux^{\K_I,c})} & {H_*(\ux^{\K,c})}
	\arrow[from=1-1, to=1-2]
	\arrow[from=1-1, to=2-1]
	\arrow[from=1-2, to=2-2]
	\arrow[from=2-1, to=2-2]
\end{tikzcd}\end{equation} 
(we use that $(\K_J)^f=(\Kf)_J$ for any $J\subseteq[m]$ and $(\K_J)_I=\K_I$ for $I\subseteq J$.)

Lemma ~\ref{lem:breakdowncolimit} then implies that there is an isomorphism \[H_*(\Omega \ux^{\K,c}) \cong \colim_{I \in \Kf} H_*(\Omega \ux^{K_I,c}).\] More precisely, for each $I \in \Kf$, the inclusions $H_*(\Omega\ux^{\K_I}) \to H_*(\Omega \ux^{\K})$ define an isomorphism $\colim_{I \in \Kf} H_*(\Omega \ux^{K_I,c}) \to H_*(\Omega \ux^{\K,c})$. Naturality with respect to simplicial inclusions then follows by construction.
\end{proof}

As a corollary, we obtain a colimit decomposition of Yoneda algebras for certain monomial rings. Recall the \emph{Stanley--Reisner ring} $\k[\K] := \k[v_1,\cdots,v_m]/(v^I=0, \: I \notin \K),$ $|v_i| = 2$, corresponding to a simplicial complex $\K$ and a ring $\k$. For $\underline{d}=(d_1,\dots,d_m)$, consider the \emph{generalised exterior Stanley--Reisner ring}
$$\Lambda(\K;\underline{d}):=S(v_1,\dots,v_m)/(v_1^2=\dots=v_m^2=0;~v^I=0,I\notin\K),\quad |v_i|=d_i,$$
where $S(v_1,\dots,v_m)$ is the free graded-commutative algebra over $\k$. As special cases, the \emph{exterior Stanley--Reisner ring} is $\Lambda(\K;\underline{1})=\Lambda[v_1,\dots,v_m]/(v^I:I\notin\K)$, and $\Lambda(\K;\underline{2})=\k[\K]/(v_1^2,\dots,v_m^2)$ is the \emph{truncated Stanley--Reisner ring.}
\begin{prp}
\label{prp:S^K model}
Let $\K$ be a simplicial complex on $[m]$, and let $\k$ be a field.
Consider $X=\ux^\K$, where $X_i=S^{d_i}$, $d_i\geq 2$. There are natural isomorphisms of algebras
$$H^*(X;\k)\cong\Lambda(\K,\underline{d}),\quad H_*(\Omega X;\k)\cong\Ext_{\Lambda(\K,\underline{d})}(\k,\k).$$
\end{prp}
\begin{proof}
The first isomorphism is from \cite[Theorem 2.34]{bbcg}. For the second isomorphism, by \cite[Theorem 4.16]{zhuravleva}, $H_*(\Omega X;\k)\cong H(\Omega_*[H_*(X;\k)])$, where $\Omega_*$ is the cobar construction. It follows that $H_*(\Omega X;\k)\cong\Ext_{H^*(X;\k)}(\k,\k)$ by dualization (see e.g. \cite[Proposition 2.6]{vylegzhanin22}).
\end{proof}
\begin{crl}
\label{cor:ext}
Let $\K$ be a simplicial complex on $[m]$, and let $\k$ be a field. Then the natural maps
\begin{enumerate}
    \item $\colim_{I\in\Kf}\Ext_{\k[\K_I]}(\k,\k)\to\Ext_{\k[\K]}(\k,\k)$,
    \item $\colim_{I\in\Kf}\Ext_{\Lambda(\K_I;\underline{d})}(\k,\k)\to\Ext_{\Lambda(\K;\underline{d})}(\k,\k)$
\end{enumerate}
are isomorphisms of algebras. Moreover, (1) is an isomorphism of Hopf algebras.
\end{crl}
\begin{proof}
We identify these algebras with $H_*(\Omega\ux^\K;\k)$ for certain $X_i$, and apply Theorem \ref{thm:uxk_colimit_decomposition}.

For (1), we have $\Ext_{\k[\K]}(\k,\k)\cong H_*(\Omega\djk;\k)\cong H_*(\Omega\ux^\K)$ for $X_i=\CC P^\infty$ by \cite[Proposition 8.4.10]{BP15}. This is an isomorphism of Hopf algebras by \cite[Proposition 6.5]{franz-hga}.

For (2), $\Ext_{\Lambda(\K,\underline{d})}(\k,\k)\cong H_*(\Omega\ux^\K;\k)$ for $X_i=S^{d_i}$ by Proposition \ref{prp:S^K model}. 
\end{proof}

An associative algebra $A$ over a field $\k$ is a \emph{K2 algebra} if the Yoneda algebra $\Ext_A(\k,\k)$ is multiplicatively generated by $\Ext_A^1(\k,\k)$ and $\Ext_A^2(\k,\k)$. This notion was introduced by Cassidy and Shelton \cite{cassidy-shelton} as a natural generalisation of the class of Koszul algebras. Generalising \cite[Corollary 5.9]{amelotte-briggs} and \cite[Theorem 1.3]{conner-shelton}, we show that a large class of Stanley--Reisner rings has this property.
\begin{thm}
\label{thm:new-k2}
Let $\k$ be a field. Let $\K$ be a simplicial complex such that, for any $I\in\Kf$, the $\k$-module $\H_*(\K_I;\k)$ is generated by classes of missing faces (i.e. all 1-neighbourly full subcomplexes of $\K$ are HMF). Then $\k[\K]$ is a K2 algebra.
\end{thm}
\begin{proof}
Recall the isomorphism $H_*(\Omega\djk;\k)\cong\Ext_{\k[\K]}(\k,\k)$. Since $u_i\in\Ext^1_{\k[\K]}(\k,\k)$ for all $i=1,\dots,m$ and $w_I\in\Ext^2_{\k[\K]}(\k,\k)$, it is sufficient to show that the algebra $H_*(\Omega\djk;\k)$ is multiplicatively generated by the elements $u_i$ and $w_I$.

If $\K$ is 1-neighbourly, then $\Kf=\Delta_{[m]}$, so $\K$ is HMF over $\k$. By part (4) of Proposition \ref{prp:HMF-without-torsion-additive-basis}, $H_*(\Omega\djk;\k)$ is multiplicatively generated by the elements $u_i$ and $w_I$, and the result follows.

In the general case, the classes $u_i$ and $w_I$ are natural with respect to the inclusions of simplicial complexes, and $H_*(\Omega\djk)=\colim_J H_*(\Omega\DJ_{\K_J})$ is generated by the generators of $H_*(\Omega\DJ_{\K_J})$. The result then follows from Theorem \ref{thm:uxk_colimit_decomposition} 
\end{proof}

\subsection*{Presentations}
\label{sec:Presentations}

In this section, we apply Theorem \ref{thm:uxk_colimit_decomposition} to give presentations of the loop homology of Davis-Januskiewicz spaces associated to certain classes of simplicial complexes. We first show that presentations of certain colimits of associative algebras can be obtained by taking unions of presentations. We fix a principal ideal domain $\k$, and denote by $T(X)$ the free graded associative $\k$-algebra generated by a set of homogeneous elements $X$.

\begin{lmm}
\label{lmm:colimit-of-presentations}
Let $\L$ be a simplicial complex on $[m]$; suppose that for each $I\in\L$, there is a set of homogeneous elements $X_I$ and a set of homogeneous elements $R_I\subseteq T(\bigsqcup_{J\subseteq I}X_I)$. For $I\in\L$, consider the algebra $A_I:=T(\bigsqcup_{J\subseteq I}X_J)/(\bigsqcup_{J\subseteq I}R_J)$ (over a principal ideal domain $\k$). Then there is a natural isomorphism
$$\colim_{I\in\L} A_I\cong T(\bigsqcup_{J\in\L} X_J)/(\bigsqcup_{J\in\L} R_J),$$ where for $I \subseteq I'\in\L$, the map $A_I \to A_{I'}$ is induced by the inclusion of the respective generating sets.
\end{lmm}
\begin{proof}
Denote $\widehat{T}:=T(\bigsqcup_{J\in\L}X_J)$. The natural maps $T(\bigsqcup_{J\subseteq I}X_J)\hookrightarrow \widehat{T}$, $q_I:T(\bigsqcup_{J\subseteq I}X_J)\twoheadrightarrow A_I$ and $\varphi_I:A_I\hookrightarrow \widehat{T}/(\bigsqcup_{J\subseteq I}R_I)\twoheadrightarrow \widehat{T}/(\bigsqcup_{J\in\L} R_J)$ induce the maps
$$\widehat{T}\cong \colim_{I\in\L} T(\bigsqcup_{J\subseteq I}X_J)\overset{q}\longrightarrow \colim_{I\in\L} A_I\overset{\varphi}\longrightarrow \widehat{T}/(\bigsqcup_{J\in\L}R_J).$$ The composition $\varphi\circ q:\widehat{T}\to \widehat{T}/(\bigsqcup_{J\in\L}A_J)$ is the standard projection, hence $\varphi$ is surjective. To prove that $\varphi$ is an isomorphism, it is sufficient to show that $\bigsqcup_{J\in\L} R_J\subset \Ker(q)$.

For each $I\in\L$, we have a commutative diagram
$$\xymatrix{
R_I\ar[rd]
\ar[r]^-{incl}
&
T(\bigsqcup_{J\subseteq I}X_J)\ar[r]^-{incl}
\ar[d]^-{q_I}
&
\widehat{T}\ar[d]^-{q}\\
&
A_I\ar[r]
&
\colim_{I\in\L} A_I,
}$$
and the diagonal arrow is zero by the definition of $A_I$. Hence $R_I\subset\Ker(q)$ for all $I\in\L$. It follows that $\bigsqcup_{I\in\L}R_I\subset\Ker(q)$, as required.
\end{proof}
\begin{thm}
\label{thm:hmf-presented-presentation-upgrade}
Let $\K$ be a simplicial complex on $[m]$ and $\k$ be a principal ideal domain with the following properties for each $J\in\Kf$:
\begin{enumerate}
    \item the complex $\K_J$ is HMF-presented over $\k$;
    \item $H_*(\K_J;\k)$ is a free $\k$-module.
\end{enumerate}
Then $H_*(\Omega\djk;\k)$ is a free $\k$-module, and there is an isomorphism of Hopf algebras
$$H_*(\Omega\djk;\k)\cong T(u_1,\dots,u_m;~w_I,I\in\MF_{\geq 3}(\K))/(R),$$
where $u_i$, $w_I$ are primitive, and $R$ is generated by relations of three types:
\begin{enumerate}
    \item $u_i^2=0$ for all $i=1,\dots,m$; $[u_i,u_j]=0$ if $\{i,j\}\in\K$;
    \item $[u_i,w_I]=0$ if $i\in I$;
    \item $\sum_{\begin{smallmatrix}i\in L:\\L\setminus i\notin\K\end{smallmatrix}}[u_i,w_{L\setminus i}]=0$ for any $L\in\AMF_{\geq 4}(\K)$.
\end{enumerate}
\end{thm}
\begin{proof}
By Theorem \ref{thm:uxk_colimit_decomposition}, $H_*(\Omega\djk;\k) \cong \colim_{I \in \Kf}H_*(\Omega\DJ_{\K_I};\k)$. The result then follows from Theorem \ref{thm:hodj-presentation-for-torsionfree-hmf-presented} and Lemma \ref{lmm:colimit-of-presentations}.
\end{proof}
Lemma \ref{lmm:skeleta-HMF-presented} implies the following special case of Theorem \ref{thm:hmf-presented-presentation-upgrade}.
\begin{crl}
\label{crl:hodj-presentation-flag-skeleta}
Let $\K$ be the $d$-skeleton of a flag simplicial complex on $[m]$, with $1\leq d\leq m$. Let $\k$ be a principal ideal domain. Then
$$H_*(\Omega\djk;\k)\cong T(u_1,\dots,u_m;~w_I,~I\in\Kf,~|I|=d+2)/(R),$$
where $R$ is generated by relations of three types:
\begin{enumerate}
    \item $u_i^2=0$ for all $i=1,\dots,m$; $[u_i,u_j]=0$ if $\{i,j\}\in\K$;
    \item $[u_i,w_I]=0$ if $i\in I$;
    \item $\sum_{i\in L}[u_i,w_{L\setminus i}]=0$ for any $L\subseteq[m]$, $|L|=d+3$.\qed
\end{enumerate}
\end{crl}

Another special case is that of ``almost flag'' simplicial complexes.
\begin{crl}
Let $\K$ be a simplicial complex such that $\K\cup\MF_{\geq 3}(\K)$ is flag, i.e. $\K$ is obtained from a flag simplicial complex $\Kf$ by deleting a set $A=\MF_{\geq 3}(\K)$ of maximal faces. Let $\k$ be a principal ideal domain. Then
$$H_*(\Omega\djk;\k)\cong T(u_1,\dots,u_m;~w_I,~I\in A)/(R),$$
where $R$ is generated by relations of two types:
\begin{enumerate}
    \item $u_i^2=0$ for all $i=1,\dots,m$; $[u_i,u_j]=0$ if $\{i,j\}\in\K$;
    \item $[u_i,w_I]=0$ if $i\in I$.\qed
\end{enumerate}
\end{crl}

\section{Poincar\'e series for colimit}
\label{sec:PS}
We now turn our attention to additive structure of the loop homology of polyhedral products. We first prove the following lemma which is a variation of the M\"obius inversion formula.

\begin{lmm}
\label{lmm:counting_lemma}
Let $\K$ be a simplicial complex on $[m]$, and let $\{a_J:J\subseteq[m]\}$ be a collection of elements in an abelian group such that the following property holds:
\begin{itemize}
    \item If $\K_{A\cup B}=\K_A\cup\K_B$ for some $A,B\subseteq [m]$, then $a_{A\cup B}=a_A+a_B-a_{A\cap B}$.
\end{itemize}
Then, for any $J\subseteq[m]$, 
$$a_J=\sum_{B\in\Kf_J}(-1)^{|B|}\sum_{I\subseteq B}(-1)^{|I|}a_I=\sum_{I\in\Kf_J}(1-\chi(\lk_{\Kf_J}I))a_I.$$
\end{lmm}
\begin{proof}
We first prove the second identity.
Since \[\{(I,B):I\subseteq B,~B\in\Kf_J\}=\{(I,I\sqcup A):~I\in\Kf_J,~A\in\lk_{\Kf_J}I\},\] we obtain the following identity
$$\sum_{B\in\Kf_J}\sum_{I\subseteq B}(-1)^{|B|}(-1)^{|I|}a_I=\sum_{I\in\Kf_J}a_I\sum_{A\in\lk_{\Kf_J}I}(-1)^{|A|},$$
where the inner sum is equal to $1-\chi(\lk_{\Kf_J}I)$.

Consider the case $J\in\Kf$. Then $\Kf_J=\Delta_J$, so $\chi(\lk_{\Kf_J}I)=\chi(\Delta_{J\setminus I})=0$ if $I\subsetneq J$, and $\chi=1$ if $I=J$. Hence $\sum_{I\in\Kf_J}(1-\chi(\lk_{\K_J}I))a_I=\sum_{I\subseteq J}\delta_{I,J}a_I=a_J$, as required. 

In the general case, we prove the first identity by induction on $|J|$. 
Suppose that the formula holds for any $I\subsetneq J$, and consider $J$. If $J\in\Kf$, the result holds by the argument above. Otherwise $J=A\cup B$ and $\K_J=\K_A\cup\K_B$ for some $A,B\subsetneq J$ (see \cite[Lemma 4.4]{S} for example). Since $\K_A\cap \K_B=\K_{A\cap B}$, we have $\Kf_A\cup\Kf_B=\Kf_J$, $\Kf_A\cap \Kf_B=\Kf_{A\cap B}$, so
$$\sum_{P\in\Kf_A}(-1)^{|P|}\sum_{I\subseteq P}(-1)^{|I|}a_I+\sum_{P\in\Kf_B}(-1)^{|P|}\sum_{I\subseteq P}(-1)^{|I|}a_I-\sum_{P\in\Kf_{A\cap B}}(-1)^{|P|}\sum_{I\subseteq P}(-1)^{|I|}a_I$$ $$=\sum_{P\in\Kf_J}(-1)^{|P|}\sum_{I\subseteq P}(-1)^{|I|}a_I$$
by the inclusion-exclusion formula.
The left hand side is equal to $a_A+a_B-a_{A\cap B}$ by the inductive assumption, so it is equal to $a_{A\cup B}=a_J$ by assumption. The right hand side is the required formula for $a_J$. Hence, the inductive step is complete.
\end{proof}

This allows us to prove Theorem \ref{thm:intropoincare_series_reduction_to_neighbourly} which reduces the calculation of the Poincar\'e series of $H_*(\Omega \ux^\K;\k)$ down to the calculation of $H_*(\Omega \ux^{\K_I};\k)$ for all $I \in \Kf$. Recall that, for a graded vector space $V=\oplus_{i \geq 0}V_i$ of finite type, the \emph{Poincar\'e series} of $V$ is defined as the power series $P(V)=P(V;t):= \sum_{i \geq 0} \mathrm{dim}(V_i)t^i \in \mathbb{Z}[[t]]$. It is clear that $P(V\oplus V')=P(V)+P(V)$, $P(V\otimes V')=P(V)\cdot P(V')$. The following result also holds for weighted polyhedral products, but for simplicity and our later applications, we state and prove the result for ordinary polyhedral products.
\begin{thm}
\label{thm:poincare_series_reduction_to_neighbourly}
Let $\k$ be a principal ideal domain and let $X_1,\dots,X_m$ be simply connected pointed spaces of finite type. Let $\K$ be a simplicial complex, and suppose $H_*(\Omega\ux^\K;\k)$ is a free $\k$-module. Then we have the following identity of Poincar\'e series.
\begin{align*}
    1/P(H_*(\Omega\ux^\K;\k))&=\sum_{I\in\Kf} (1-\chi(\lk_{\Kf} I))/P(H_*(\Omega\ux^{\K_I};\k))\\
    &=\sum_{B\in\Kf}(-1)^{|B|}\sum_{I\subseteq B}(-1)^{|I|}/P(H_*(\Omega\ux^{\K_I};\k).
\end{align*}
\end{thm}
\begin{proof}
We apply Lemma \ref{lmm:counting_lemma} to the elements $a_J:=1/P(H_*(\Omega\ux^{\K_J};\k);t)$ of $\ZZ[[t]]$. It is sufficient to verify the following property:
\begin{itemize}
    \item If $\K_{A\cup B}=\K_A\cup\K_B$, then $a_{A\cup B}=a_A+a_B-a_{A\cap B}$.
\end{itemize}
Indeed, by Lemma ~\ref{lem:pushout}, there is a pushout \[\begin{tikzcd}
	{\ux^{\K_{A \cap B}}} & {\ux^{\K_A}} \\
	{\ux^{\K_B}} & {\ux^\K.}
	\arrow["{f_A}", from=1-1, to=1-2]
	\arrow["{f_B}", from=1-1, to=2-1]
	\arrow[from=1-2, to=2-2]
	\arrow[from=2-1, to=2-2]
\end{tikzcd}\] Lemma ~\ref{lem:retraction} implies that $f_A$ and $f_B$ have left inverses, and so the result follows from part (3) of Theorem \ref{thm:general-theorem-for-pushout}. 
\end{proof}

\subsection{Properties of Backelin--Berglund polynomials}

Let $\K$ be a simplicial complex and $\k$ be a field.
In \cite[Definition 4.2, Definition 4.4]{SV}, the \emph{reflected Backelin--Berglund polynomial} $\widehat{bb}_{\K,\k}(t)\in\ZZ[t]$ are introduced, which are related to the multigraded Poincar\'e series of the $\k$-module $\Tor^{\k[\K]}(\k,\k)$. In \cite[Theorem 6.22]{SV}, it is shown that the additive structure of the loop homology of simply connected polyhedral products can be expressed in terms of these polynomials. In particular, the following result was proved.
\begin{theorem}
\label{thm:PSforuxa}
    Let $(X_1,A_1),\dots,(X_m,A_m)$ be $CW$-pairs and $\K$ be a simplicial complex on $[m]$ such that $\uxa^\K$ is simply connected and of finite type. Let $\k$ be a field. Denote by $G_i$ the homotopy fibre of $A_i \rightarrow X_i$. 
    Then
    $$\frac{1}{P(H_*(\Omega\uxa^\K;\k);t)}=\frac{\sum_{J\subseteq[m]}\widehat{bb}_{\K_J,\k}(t)\cdot\prod_{j\in J}P(\H_*(G_j;\k);t)}{\prod_{i=1}^mP(H_*(\Omega X_i;\k);t)}\in\ZZ[[t]].$$
\end{theorem}

The aim of the rest of this section is to calculate the polynomials $\widehat{bb}_{\K,\k}$ (and hence the Poincar\'e series of loop homology of polyhedral products) for certain classes of simplicial complexes. To do this, we first reduce the study of these polynomials to the study of the $\ZK$ and $\djk$ case. We will need the following version of the M\"obius formula.
\begin{lmm}
\label{lmm:mobius}
Let $\{a_J:H\subseteq [m]\}$ be elements in an abelian group, and $b_I:=\sum_{J\subseteq I}a_J$. Then 
$(-1)^ma_{[m]}=\sum_{I\subseteq [m]}(-1)^{|I|}b_I.$\qed
\end{lmm}
\begin{prp}
\label{prp:formula-for-bb}
For a simplicial complex $\K$ and a field $\k$, we have
\begin{gather*}
1/P(H_*(\Omega\ZK;\k);t)=\sum_{J\subseteq[m]}\widehat{bb}_{\K_J}(t)t^{|J|},\quad 
1/P(H_*(\Omega\djk;\k);t)=(1+t)^{-m}\sum_{J\subseteq[m]}\widehat{bb}_{\K_J}(t)t^{|J|},
\end{gather*}
\begin{equation}
\label{eqn:formula-for-bb}
(-t)^m\cdot \widehat{bb}_{\K,\k}(t)=\sum_{I\subseteq[m]}(-1)^{|I|}/P(H_*(\Omega\Z_{\K_I};\k);t)\in\ZZ[[t]].
\end{equation}
\end{prp}
\begin{proof}
The first two identities follow from \cite[Proposition 4.8]{SV}. For each $\K_I$, we therefore obtain $1/P(H_*(\Omega\Z_{\K_I};\k);t)=\sum_{J\subseteq I}\widehat{bb}_{\K_J}(t)t^{|J|}$.
Now \eqref{eqn:formula-for-bb} follows from Lemma \ref{lmm:mobius}.
\end{proof}

\begin{prp}
\label{prp:bb-additivity}
There is an equality $\widehat{bb}_{\K_A\cup \K_B,\k}(t)=\widehat{bb}_{\K_A,\k}(t)+\widehat{bb}_{\K_B,\k}(t)-\widehat{bb}_{\K_{A\cap B},\k}(t)$.
\end{prp}
\begin{proof}
Without loss of generality, write $\K=\K_A\cup \K_B$. Then for each $I\subseteq[m]$ the pushout $\K=\K_A\cup_{\K_{A\cap B}}\K_B$ restricts to the pushout $\K_I=\K_{A\cap I}\cup_{\K_{A\cap B\cap I}}\K_{B\cap I}$.

For $J\subseteq[m]$, denote $\varphi_J:=(-1)^{|J|}/P(H_*(\Omega\Z_{\K_J};\k);t)\in\ZZ[[t]]$. Then $\sum_{I\subseteq J}\varphi_I=(-t)^{|J|}\widehat{bb}_{\K_J}(t)$ by \eqref{eqn:formula-for-bb}. On the other hand, $1/P(H_*(\Omega\DJ_{\K_J};\k);t)=(-1-t)^{-|J|}\varphi_J$ by Proposition \ref{prp:formula-for-bb}. Applying part (3) of Theorem \ref{thm:general-theorem-for-pushout} to the pushout
\[\begin{tikzcd}
	{\DJ_{\K_{A\cap B\cap I}}} & {\DJ_{\K_{A\cap I}}} \\
	{\DJ_{\K_{B\cap I}}} & {\DJ_{\K_I}}
	\arrow[from=1-1, to=1-2]
	\arrow[from=1-1, to=2-1]
	\arrow[from=1-2, to=2-2]
	\arrow[from=2-1, to=2-2]
\end{tikzcd}\]
from Lemma \ref{lem:pushout}, we obtain:
\begin{gather*}
(-1-t)^{-|I|}\varphi_I=(-1-t)^{-|A\cap I|}\varphi_{A\cap I}+(-1-t)^{-|B\cap I|}\varphi_{B\cap I}-(-1-t)^{-|A\cap B\cap I|}\varphi_{A\cap B\cap I},\\\varphi_I=(-1-t)^{|I\setminus A|} \varphi_{A\cap I}+(-1-t)^{|I\setminus B|} \varphi_{B\cap I}-(-1-t)^{|I\setminus (A\cap B)|}\varphi_{A\cap B\cap I}.\end{gather*}
Now we take the sum over all $I\subseteq[m]$. Note that
$$\sum_{I\subseteq[m]}(-1-t)^{|I\setminus A|} \varphi_{A\cap I}=\sum_{I'\subseteq [m]\setminus A}(-1-t)^{|I'|} \sum_{I''\subseteq A}\varphi_{I''}=(-t)^{m-|A|} (-t)^{|A|}\widehat{bb}_{\K_A,\k}(t)$$ $$=(-t)^m\widehat{bb}_{\K_A,\k}(t).$$
Similar formulas hold for $B$ and $A\cap B$, so we obtain:
$$(-t)^m\widehat{bb}_{\K,\k}(t)=(-t)^m\widehat{bb}_{\K_A,\k}(t)+(-t)^{m}\widehat{bb}_{\K_B,\k}(t)-(-t)^{m}\widehat{bb}_{\K_{A\cap B},\k}(t)$$
as required.
\end{proof}
\begin{crl}
\label{crl:bb-main-formula}
For any simplicial complex $\K$, we have $$\widehat{bb}_{\K,\k}(t)=\sum_{I\in\Kf}(1-\chi(\lk_{\Kf}I))\widehat{bb}_{\K_I,\k}(t)=\sum_{J\in\Kf}(-1)^{|J|}\sum_{I\subseteq J}(-1)^{|I|}\widehat{bb}_{\K_I,\k}(t).$$
\end{crl}
\begin{proof}
The result follows from Lemma \ref{lmm:counting_lemma} applied to the elements $a_J=\widehat{bb}_{\K_J}(t)\in\ZZ[t]$, where the property $a_{A\cup B}=a_{A}+a_B-a_{A\cap B}$ holds by Proposition \ref{prp:bb-additivity}.
\end{proof}

Since $\K_I$ is a $1$-neighbourly complex for any $I\in\Kf$, the above proposition shows that it is sufficient to compute the polynomials $\widehat{bb}_{\K,\k}(t)$ for \emph{$1$-neighbourly} complexes.

\begin{prp}
\label{prp:bb-for-Golod}
Let $\K$ be a simplicial complex which is Golod over a field $\k$. Then
$$\widehat{bb}_{\K,\k}(t)=1-P(H_*(\K;\k);t).$$
\end{prp}
\begin{proof}
By \cite[Theorem 1]{panov-limonchenko}, $\K$ is Golod over $\k$ if and only if $H_*(\Omega\ZK;\k)$ is a free associative algebra. Since a retract of a free algebra is free, $\K_J$ is Golod over $\k$ for any $I\subseteq[m]$.

By \cite[Theorem 1(d)]{panov-limonchenko},
$$1/P(H_*(\Omega\ZK;\k);t)=1-\sum_{J\neq\varnothing,J\subseteq[m]}P(\H_*(\K_J;\k);t)t^{|J|}=\sum_{J\subseteq[m]}(1-P(H_*(\K_J;\k);t))t^{|J|}.$$
Similarly, $1/P(H_*(\Omega\Z_{\K_I};\k);t)=\sum_{J\subseteq I}(1-P(H_*(\K_J;\k);t))t^{|J|}$ for each $I\subseteq[m]$. By Lemma \ref{lmm:mobius},
$$(-1)^m\cdot (1-P(H_*(\K;\k);t))t^m=\sum_{I\subseteq[m]}(-1)^{|I|}/P(H_*(\Omega\Z_{\K_I};\k);t).$$
The right hand side is equal to $(-t)^m\widehat{bb}_{\K,\k}(t)$ by Proposition \ref{prp:formula-for-bb}.
\end{proof}
\begin{crl}
If $\K$ is HMF over a field $\k$, then $\widehat{bb}_{\K,\k}(t)=1-P(H_*(\K;\k);t)$.   
\end{crl}
\begin{proof}
By part (2) of Proposition \ref{prp:HMF-without-torsion-additive-basis}, $H_*(\Omega\ZK;\k)$ is a free associative algebra, so $\K$ is Golod over $\k$.
\end{proof}

Recall that $\K$ is \emph{$q$-neighbourly} if $\sk_q\K=\sk_q\Delta_{[m]}$. It is known that $\K$ is Golod if $2q\geq \dim \K$ (see \cite[Theorem 10.9]{fat-wedge}); in particular, $\widehat{bb}_{\K,\k}(t)=1-P(H_*(\K;\k);t)$ in that case. We obtain a partial result under milder conditions on $q$.
\begin{lmm}
\label{lem:PShighlyconn}
Let $X$ be a $k$-connected CW-complex of finite type, $k\geq 1$. Then we have that $1/P(H_*(\Omega X;\k);t)=1-t^{-1}P(\H_*(X;\k);t)+O(t^{2k}).$
\end{lmm}
\begin{proof}
The connectivity of $X$ implies $P(\H_*(X;\k);t)=t^{k+1}f(t)$ for some $f\in\ZZ[[t]]$. It is well-known that $\H_*(X;\k)\cong \H_{*-1}(\Omega X;\k)$ for $*\leq 2k$, so $P(H_*(\Omega X;\k);t)=1+t^kf(t)+O(t^{2k})$. It follows that $1/P(H_*(\Omega X;\k);t)=1-t^kf(t)+O(t^{2k})$, so the result follows.
\end{proof}
\begin{prp}
\label{prp:bb-golod-approx}
Let $\K$ be a $q$-neighbourly simplicial complex on $[m]$, where $4q\geq m-2$. Then $\widehat{bb}_{\K,\k}(t)=1-P(H_*(\K;\k);t)+O(t^{4q+2-m})$.
\end{prp}
\begin{proof}
By \cite[Proposition 4.3.5(b)]{BP15}, the moment-angle complex $\ZK$ is $(2q+1)$-connected, and so $1/P(H_*(\Omega\ZK;\k);t)=1-t^{-1}P(\H_*(\ZK;\k);t)+O(t^{4q+2})$ by Lemma \ref{lem:PShighlyconn}. Similarly, for any $I\subseteq[m]$, $$1/P(H_*(\Omega\Z_{\K_I};\k);t)=1-t^{-1}(P(\H_*(\Z_{\K_I};\k);t)+O(t^{4q+2})=\sum_{J\subseteq I}(1-P(H_*(\K_J;\k);t))t^{|J|}+O(t^{4q+2}).$$
Therefore, arguing as in the proof of Proposition \ref{prp:bb-for-Golod}, we obtain \[\widehat{bb}_{\K,\k}(t)t^m=(1-P(H_*(\K;\k);t))t^m+O(t^{4q+2}).\qedhere\]
\end{proof}

Finally, let $\K$ be a skeleton of a flag simplicial complex. In this case, there is a simple formula for reflected Backelin--Berglund polynomials. 
\begin{prp}
\label{prp:bb-for-skeleta}
Let $\L$ be a flag complex, and let $\K=sk_d\L$. Then $$\widehat{bb}_{\K,\k}(t)=1-\chi(\L)+(-t)^d(\chi(\L)-\chi(\K)).$$    
\end{prp}
\begin{proof}
We proceed by induction on the number of vertices.

The base case: let $\L$ be a simplex, $\L=\Delta_I$. Then $\K=\sk_d\Delta_I$ is a Golod complex, and $\K\simeq(S^d)^{\vee n_d}$ for some $n_d$. We have $\chi(\L)=1$, $\chi(\K)=1+(-1)^d n_d$ and
$$\widehat{bb}_{\K,\k}(t)=1-P(H_*(\K;\k);t)=1-(1+n_dt^d)=(-t)^d(1-\chi(\K))=1-\chi(\L)+(-t)^d(\chi(\L)-\chi(\K)),$$
as required.

The induction step: suppose that the formula holds for all proper full subcomplexes of $\L$, and $\L$ is not a simplex. We write $\K=\K_A\cup_{\K_{A\cap B}}\K_B$ for some $A,B\subsetneq [m]$. Then $\L=\L_A\cup_{\L_{A\cap B}}\L_B$. Moreover, by Proposition \ref{prp:bb-additivity} $$\widehat{bb}_{\K,\k}(t)=\widehat{bb}_{\K_A,\k}(t)+\widehat{bb}_{\K_B,\k}(t)-\widehat{bb}_{\K_{A\cap B},\k}(t).$$
By the inductive hypothesis applied to $\K_A$, $\K_B$, $\K_{A\cap B}$, $\widehat{bb}_{\K,\k}(t)$ is equal to
$$1-\Big(\chi(\L_A)+\chi(\L_B)-\chi(\L_{A\cap B})\Big)+(-t)^d\cdot\Big((\chi(\L_A)+\chi(\L_B)-\chi(\L_{A\cap B})-(\chi(\K_A)+\chi(\K_B)-\chi(\K_{A\cap B}))\Big),$$
which equals the required formula since $\chi(\K)=\chi(\K_A)+\chi(\K_B)-\chi(\K_{A\cap B})$, and the same holds for $\L$.
\end{proof}

\begin{rmk}
\label{rmk:skeletonsimplexcase}
If $\K=sk_d(\Delta_{[m]})\simeq (S^d)^{\vee n_d}$ is a $d$-skeleton of a simplex, then $\widehat{bb}_\K(t)=-n_dt^d$, where $n_d$ is computed by counting the Euler characteristic twice: $$1+(-1)^d n_d=m-\binom{m}{2}+\dots+(-1)^{d}\binom{m}{d+1}.$$ For example, if $d=1$, then $\K$ is a complete graph, and the Backelin-Berglund polynomial is equal to  $bb_\K(t)=-n_1t$, $n_1=1-m+\binom{m}{2}=(m^2-3m+2)/2$.
\end{rmk}

The following proposition is a summary of our results as well as some simple computations.
\begin{prp}
Reflected Backelin--Berglund polynomials have the following properties:
\begin{enumerate}
    \item $\deg \widehat{bb}_{\K,\k}(t)\leq m$.
    \item $\widehat{bb}_{\{\varnothing\},\k}(t)\equiv 1$; $\widehat{bb}_{\Delta_I,\k}(t)=0$ if $I\neq\varnothing$; $\widehat{bb}_{\partial\Delta_I,\k}=-t^{|I|-2}$.
    \item $\widehat{bb}_{\K\ast\L,\k}(t)=\widehat{bb}_{\K,\k}(t)\cdot\widehat{bb}_{\L,\k}(t)$.
    \item If $\K$ is a flag complex, then $\widehat{bb}_{\K,\k}(t)=1-\chi(\K)$ is a constant.
    \item If $\K$ is the $d$-skeleton of a flag simplicial complex $\L$ on $[m]$ with $1 \leq d \leq m$, then we have $\widehat{bb}_{\K,\k}(t)=1-\chi(\L)+(-t)^d(\chi(\L)-\chi(\K))$.
    \item If $\K$ is Golod over $\k$, then $\widehat{bb}_{\K,\k}(t)=1-P(H_*(\K;\k);t)$.
    \item If $\K$ is $q$-neighbourly, $4q\ge m-2$, then $\widehat{bb}_{\K,\k}(t)=1-P(H_*(\K;\k);t)+O(t^{4q+2-m})$. 
    \item $\widehat{bb}_{\K,\k}(t)=\sum_{I\in\Kf}(1-\chi(\lk_{\Kf}I))\widehat{bb}_{\K_I,\k}(t)=\sum_{J\in\Kf}(-1)^{|J|}\sum_{I\subseteq J}\widehat{bb}_{\K_I,\k}(t)$.
\end{enumerate}
\end{prp}
\begin{proof}
Statement (1) follows from \cite[Proposition 4.3 + Definition 4.4]{SV}. Statement (2) follows from the homotopy equivalences $\Z_{\Delta_I}\simeq\pt$, $\Z_{\partial\Delta_I}\simeq S^{2|I|-1}$ and the Bott--Samelson theorem (or from (6)). For (3), note that $\Z_{\K\ast\L}=\ZK\times\Z_{\L}$. Moreover: if $\K$ is on the vertex set $V$, $\L$ is on the vertex set $W$ and $A\subseteq V$, $B\subseteq W$ then $(K\ast\L)_{A\sqcup B}=\K_{A}\ast \L_{B}$. It follows that $P(H_*(\Omega\Z_{(\K\ast\L)_{A\sqcup B}});t)=P(H_*(\Omega\Z_{\K_{A}});t)\cdot P(H_*(\Omega\Z_{\L_B});t)$, so $\widehat{bb}_{\K\ast\L,\k}(t)=\widehat{bb}_{\K,\k}(t)\cdot\widehat{bb}_{\L,\k}(t)$ by \eqref{eqn:formula-for-bb}.

Statement (4) is \cite[Proposition 4.9]{SV}, (5) is Proposition \ref{prp:bb-for-skeleta}, (6) is Proposition \ref{prp:bb-for-Golod}, (7) is Proposition \ref{prp:bb-golod-approx} and (8) is Corollary \ref{crl:bb-main-formula}.
\end{proof}

This allows to compute the Poincar\'e series for the loop homology of polyhedral products in certain cases.
\begin{prp}
Let $\K$ be a simplicial complex such that, for all $I\in\Kf$, the complex $\K_I$ is Golod over $\k$. Then
$$\widehat{bb}_{\K,\k}(t)=\sum_{I\in\Kf}(1-\chi(\lk_{\Kf}I))(1-P(H_*(\K_I;\k);t)=\sum_{J\in\Kf}(-1)^{|J|}\sum_{I\subseteq J}(-1)^{|I|}(1-P(H_*(\K_I;\k);t)),$$
$$1/P(H_*(\Omega\ZK;\k);t)=\sum_{L\in\Kf}(-t)^{|L|}(1+t)^{m-|L|}\sum_{I\subseteq L}(1-P(\H_*(\K_I;\k);t)),$$
$$1/P(H_*(\Omega\djk;\k);t)=\sum_{L\in\Kf}(-t/(1+t))^{|L|}\sum_{I\subseteq L}(1-P(\H_*(\K_I;\k);t)).$$
In particular, this holds in the following cases:
\begin{itemize}
    \item $\K_I$ is HMF over $\k$, for all $I\in\Kf$;
    \item $\dim\K\leq 2$;
    \item $\K$ is a skeleton of a flag complex.
\end{itemize}
\end{prp}
\begin{proof}
By Corollary \ref{crl:bb-main-formula}, there is an equality \[\widehat{bb}_{\K,\k}(t)=\sum_{I\in\Kf}(1-\chi(\lk_{\Kf}I))\widehat{bb}_{\K_I,\k}(t)=\sum_{J\in\Kf}(-1)^{|J|}\sum_{I\subseteq J}(-1)^{|I|}\widehat{bb}_{\K_I,\k}(t).\] The first formula then holds by Proposition \ref{prp:bb-for-Golod}. The Poincar\'e series formula for $H_*(\Omega \ZK;\k)$ and $H_*(\Omega \djk;\k)$ hold by Proposition \ref{prp:formula-for-bb}.

For each $I \in \Kf$, $H_*(\Omega\mathcal{Z}_{\K_I};\k)$ is a free associative algebra when $\K_I$ is HMF over $\k$ by part (2) of Proposition \ref{prp:HMF-without-torsion-additive-basis}. When $K_I$ is a $2$-dimensional simplicial complex, or the $d$-skeleton of a flag complex, then $\mathcal{Z}_{\K}$ is homotopy equivalent to a suspension by \cite[Theorem 10.9]{fat-wedge} and \cite[Theorem 1]{Po66} respectively (in the second case $I\in\Kf$ implies that $\K_I$ is a 1-skeleton of a simplex). Therefore, the Bott-Samelson theorem implies $H_*(\Omega\mathcal{Z}_{\K_I};\k)$ is a free associative algebra.
\end{proof}

\begin{prp}
\label{prp:Dn-computation}
Let $\K$ be a simplicial complex such that
\begin{itemize}
    \item $H_*(\ZK;\ZZ)$ is torsion-free;
    \item For all $I\in\Kf$, $H_*(\K_I;\ZZ)$ is generated by classes of missing faces. 
\end{itemize}
Then
$$\Omega\ZK\simeq\prod_{n\geq 2}(\Omega S^n)^{\times D_n},$$
where the numbers $D_n\geq 0$ are determined by the identity
$$\prod_{n\geq 2}(1-t^{n-1})^{D_n}=\sum_{L\in\Kf}(-t)^{|L|}(1+t)^{m-|L|}\sum_{I\subseteq[m]}(1-P(H_*(\K_I;\k);t)).$$
\end{prp}
\begin{proof}
By assumption, $\K_I$ is HMF over $\ZZ$ for each $I\in\Kf$, and $H_*(\Z_{\K_I};\ZZ)$ is torsion-free. By \cite[Corollary 4.6]{amelotte-briggs}, $\Z_{\K_I}$ is homotopy equivalent to a wedge of spheres. 
 It follows from \cite[Theorem 5.5]{S1} and \cite[Lemma 6.1]{vylegzhanin25} that $\Omega\ZK$ is homotopy equivalent to a product of loops on spheres. The numbers $D_n$ are computed by comparing Poincar\'e series, as in \cite[Theorem 1.2]{vylegzhanin25}.
\end{proof}

By Lemma \ref{lmm:skeleta-HMF-presented}, Proposition \ref{prp:Dn-computation} applies when $\K$ is the $d$-skeleton of a flag complex, and $\K$ is a graph. These cases were studied in \cite{S}, and the following results are applications of Propositions \ref{prp:formula-for-bb} and \ref{prp:bb-for-skeleta}.

\begin{crl}
Let $\L$ be a flag complex, and let $\K=\sk_d\L$. Then
$$1/P(H_*(\Omega\ZK;\k);t)=\sum_{J\subseteq[m]}(1-\chi(\L_J)+(-t)^d(\chi(\L_J)-\chi(\K_J)))\cdot t^{|J|},$$
\[1/P(H_*(\Omega\ux^\K;\k);t)=\frac{\sum_{J\subseteq[m]}(1-\chi(\L_J)+(-t)^d(\chi(\L_J)-\chi(\K_J))\prod_{i\in J}P(\H_*(\Omega X_j;\k);t)}{\prod_{i=1}^m P(H_*(\Omega X_i;\k);t)}.\qed\]
\end{crl}
\begin{crl}
Let $\Gamma$ be a graph on $[m]$. Consider the flag simplicial complex $\L:=\Gamma^f$. Then
\[1/P(H_*(\Omega\Z_\Gamma;\k);t)=\sum_{J\subseteq[m]}(1-\chi(\L_J)-t(\chi(\L_J)-\chi(\Gamma_J)))\cdot t^{|J|}.\qed\]
\end{crl}

\section{Weighted polyhedral products associated to graphs}
\label{sec:weighted}

In this section, we generalise some of the previous results to weighted polyhedral products. In particular, we prove general results related to torsion in the loop homology and loop space decompositions of weighted polyhedral products. We then give a presentation of $H_*(\Omega \ux^{\K,c};\k)$, when each $X_i$ is a simply connected sphere, $\K$ is a graph and $\k$ is a field. Moreover, we complete the picture of the homotopy theory of $\Omega \ux^{\K,c}$ when $\K$ is a tree, by giving a decomposition of this space up to homotopy equivalence. When each $X_i$ is a simply connected sphere $S_i=S^{d_i}$, $d_i\geq 2$, we denote the corresponding polyhedral product by $\us^\K$.

\subsection*{Torsion and loop space decompositions}

Recall that for a space $X$, $\tau(X)$ is the set of primes appearing as $p$-torsion in $H_*(X;\ZZ)$. Theorem \ref{thm:poincare_series_reduction_to_neighbourly} allows one to study $\tau(\Omega\ux^{\K,c})$ by comparing Poincar\'e series with different coefficients. We give a direct argument below. 
\begin{thm}
\label{thm:torsion-reduction-to-neighbourly}
There is an equality of sets $\tau(\Omega\ux^{\K,c})=\bigcup_{I\in\Kf}\tau(\Omega\ux^{\K_I,c})$.
\end{thm}
\begin{proof}
We argue similarly to the proof of Theorem \ref{thm:uxk_colimit_decomposition}.

If $\K$ is 1-neighbourly, then the right hand side is equal to $\bigcup_{I\subseteq[m]}\tau(\Omega\ux^{\K_I,c})$. Since $\ux^{\K_I,c}$ is a retract of $\ux^{\K,c}$, we have $\tau(\Omega\ux^{\K_I,c})\subseteq\tau(\Omega\ux^{\K,c})$, so the result follows.

If $\K$ is not 1-neighbourly, then $\K=\K_A\cup\K_B$ for some $A,B\subsetneq [m]$ (see e.g. \cite[Lemma 4.4]{S}), hence,
$$\tau(\Omega\ux^{\K,c})=\tau(\Omega\ux^{\K_A,c})\cup\tau(\Omega\ux^{\K_B,c})=\bigcup_{I\in\Kf_A\cup\Kf_B}\tau(\Omega\ux^{\K_I,c})=\bigcup_{I\in\Kf}\tau(\Omega\ux^{\K_I,c}),$$
where the first identity follows from part (1) of Theorem \ref{thm:general-theorem-for-pushout}, the second follows from the inductive assumption and the third follows from the identity $\Kf_A\cup\Kf_B=(\K_A\cup\K_B)^f=\Kf.$
\end{proof}

 One can also give a similar reduction of the homotopy type of the loops on certain weighted polyhedral products. Recall that $\prod(\mathcal{P} \cup \mathcal{T})$ is the collection of topological spaces which are homotopy equivalent to a product of spheres, loops on spheres, and indecomposable torsion spaces. The following result follows by arguing identically to \cite[Theorem 5.3]{S1} using Theorem \ref{thm:decomposition-for-pushout}.

\begin{theorem}
\label{thm:homotopy-decomp-reduction-to-neighbourly}
If $\Omega\ux^{\K_I,c}\in\prod(\mathcal{P\cup T})$ for all $I\in\Kf$, then $\Omega\ux^{\K,c}\in\prod(\mathcal{P\cup T})$. $\qed$
\end{theorem}

\subsection*{Presentations}
First, we state a special case of a result of Lemaire \cite[4.3.1]{lemaire}, which gives a presentation of $H_*(\Omega (\us^\K)$ when $\K$ is the $1$-skeleton of a simplex. This will allow us to give a presentation when $\K$ is any graph.

\begin{thm}[{\cite[Theoreme 3.1.6]{lemaire}}]
\label{thm:lemaire-graph}
Suppose that $X=\us^{\sk_1\Delta_{[m]}}$, where $S_i=S^{d_i}$. Let $\k$ be a field. Then there is an isomorphism of algebras
$$H_*(\Omega X;\k)\cong T(b_i,~1\leq i\leq m;~w_I,~I\subseteq[m],~|I|=3)/(R),$$
where $\deg b_i=d_i-1$, $\deg w_{\{i,j,k\}}=d_i+d_j+d_k-2$, and the ideal $R$ has relations of three types:
\begin{enumerate}
    \item $[b_i,b_j]=0$ for $i\neq j$;
    \item $[b_i,w_I]=0$ for $i\in I$;
    \item $\sum_{i\in J}\pm [b_i,w_{J\setminus i}]=0$ for $|J|=4$.
\end{enumerate}
Geometrically, $b_j$ is adjoint to the sphere inclusion $\iota_j:S^{d_j}\hookrightarrow X$ and $w_{\{i,j,k\}}$ is adjoint to the canonical higher Whitehead product $[\iota_i,\iota_j,\iota_k]:S^{d_i+d_j+d_k-2}\to X$.\qed
\end{thm}

\begin{crl}
\label{crl:graph-product-of-spheres}
Suppose that $X=\us^{\Gamma}$, where $S_i=S^{d_i}$ and $\Gamma$ is a graph. Let $\k$ be a field. Then there is an isomorphism of algebras
$$H_*(\Omega X;\k)\cong T(b_i,~1\leq i\leq m;~w_I,~I\in\MF(\Gamma))/(R),$$
where $\deg b_i=d_i-1$, $\deg w_{\{i,j,k\}}=d_i+d_j+d_k-2$, and the ideal $R$ has relations of three types:
\begin{enumerate}
    \item $[b_i,b_j]=0$ for $\{i,j\}\in\Gamma$;
    \item $[b_i,w_I]=0$ for $i\in I$, $I\in\MF(\Gamma)$;
    \item $\sum_{i\in J}\pm [b_i,w_{J\setminus i}]=0$ for $|J|=4$, $J\in\Gamma^f$.
\end{enumerate}
\end{crl}
\begin{proof}
By Theorem \ref{thm:uxk_colimit_decomposition}, $H_*(\Omega(\underline{S},\underline{*})^{\K});\k) \cong \colim_{I \in \Kf}H_*(\Omega(\underline{S},\underline{*})^{\K_I};\k)$. The result then follows from Theorem \ref{thm:lemaire-graph} and Lemma \ref{lmm:colimit-of-presentations}.
\end{proof}

Now we move onto the case of weighted polyhedral products $X=\us^{\K,c}$. We assume that $c(\{i\})=1$ for all $i=1,\dots,m$. Then the power sequence $c$ is determined by the positive integers $c_i^{\{i,j\}}$ for $i\in\{i,j\}\in\K$. Denote $$m_{ij}:=c_i^{\{i,j\}}\cdot c_j^{\{i,j\}}\text{ for }\{i,j\}\in\K.$$ 

Recall that the \emph{Adams-Hilton model} (see \cite{adams-hilton}, \cite[\S 4]{zhuravleva}) of a simply connected CW complex $X$ with trivial 1-skeleton is a tensor dg-algebra $AH(X)$ with multiplicative generators corresponding to cells of $X$, such that $H(AH(X))\cong H_*(\Omega X;\k)$ as algebras.
\begin{lmm}
\label{lmm:weight-graph-product-AH-model}
Let $X=\us^{\K,c}$ be a weighted polyhedral product, where $S_i=S^{d_i}$, $c(\{i\})=1$ for all $i=1,\dots,m$ and $\K$ is a graph.
Then $X$ has an Adams--Hilton model
$$AH(X)=\Big(T(b_i,~1\leq i\leq m;~b_{ij},~\{i,j\}\in\K),~d\Big),\quad d(b_i)=0,~d(b_{ij})=m_{ij}\cdot [b_i,b_j],$$
$\deg b_i=d_i-1$, $\deg b_{ij}=d_i+d_j-1$.
In particular, $\us^\K$ has an Adams--Hilton model
$$AH(\us^\K)=\Big(T(b_i,~1\leq i\leq m;\:b_{ij},~\{i,j\}\in\K),~d\Big),\quad d(b_i)=0,~d(b_{ij})=[b_i,b_j].$$
\end{lmm}
\begin{proof}
By definition, $X$ is obtained from $S^{d_1}\vee\dots\vee S^{d_m}$ by simultaneously attaching $(d_i+d_j)$-cells along the Whitehead brackets $m_{ij}[\iota_i,\iota_j]_{Wh}:S^{d_i+d_j-1}\to \bigvee_{t=1}^m S^{d_t}$, $\{i,j\}\in\K$.

Since a colimit of consistent Adams--Hilton models is an Adams--Hilton model (see \cite[Theorem 4.3(8)]{zhuravleva}), it is sufficient to consider the case $m=2$, in other words, the case where $X=(S^{d_1}\vee S^{d_2})\cup_{m_{12}[\iota_1,\iota_2]} e^{d_1+d_2}$. By linearity of the construction of the Adams--Hilton model (see the proof of \cite[Theorem 2.1]{adams-hilton}), it is sufficient to consider the case $m_{12}=1$, i.e. $X=S^{d_1}\times S^{d_2}$. Then the result follows from \cite[Theorem 4.3]{adams-hilton}.
\end{proof}

\begin{thm}
Let $X=\us^{\K,c}$ be a weighted polyhedral product, where $S_i=S^{d_i}$, $c(\{i\})=1$ for all $i=1,\dots,m$ and $\K$ is a graph. Let $\k$ be a field. Consider the graph
$$\Gamma:=\{\{i,j\}\in\K:~c_i^{\{i,j\}}\cdot c_j^{\{i,j\}}\neq 0\in\k\}.$$
Then there is an isomorphism of algebras
$$H_*(\Omega X;\k)\cong T(b_{ij}:\{i,j\}\in\K\setminus \Gamma)\ast~T(b_i,~1\leq i\leq m;~w_I,~I\in\MF(\Gamma))/(R),$$
where $$\deg b_i=d_i-1,\quad \deg b_{ij}=d_i+d_j-1,\quad\deg w_{\{i,j,k\}}=d_i+d_j+d_k-2,$$
and the ideal $R$ has relations of three types:
\begin{enumerate}
    \item $[b_i,b_j]=0$ for $\{i,j\}\in\Gamma$;
    \item $[b_i,w_I]=0$ for $i\in I\in\MF(\Gamma)$;
    \item $\sum_{i\in J}\pm[b_i,w_{J\setminus i}]=0$ for $|J|=4$, $J\in\Gamma^f$.
\end{enumerate}
\end{thm}
\begin{proof}
We compute $H_*(\Omega X;\k)$ as the homology of the dg-algebra $AH(X)$ from Lemma \ref{lmm:weight-graph-product-AH-model}.

For $\{i,j\}\in\Gamma$, we have $m_{ij}\in\k\setminus\{0\}=\k^\times$, so we can replace the multiplicative generator $b_{ij}$ with the element $b_{ij}':=b_{ij}/m_{ij}$. For $\{i,j\}\in\K\setminus \Gamma$, $m_{ij}=0$. It follows that
\begin{multline*}
AH(\ux^{\K,c})\cong \Big(T(b_i,~i\in[m];~b_{ij}',~\{i,j\}\in\Gamma;~b_{ij},~\{i,j\}\in\K\setminus\Gamma),d\Big),\\
d(b_i)=0,~d(b_{ij}')=[b_i,b_j],~d(b_{ij})=0.
\end{multline*}
This dg algebra is isomorphic to a free product $(T(b_{ij}:~\{i,j\}\in\K\setminus\Gamma),0)\ast AH(\ux^{\Gamma}),$ so we obtain an isomorphism of algebras
$$H_*(\Omega\ux^{\K,c};\k)\cong T(b_{ij}:~\{i,j\}\in\K\setminus\Gamma)\ast H_*(\Omega \ux^\Gamma;\k).$$
The result follows from Corollary \ref{crl:graph-product-of-spheres}.
\end{proof}

We conclude this section by giving a loop space decomposition of $\Omega \us^{\K,c}$ when $\K$ is a tree. To do this, we first give a decomposition in the case where $\K$ is a $1$-simplex.

\begin{lemma}
\label{lem:1simplexdecomp}
Let $X_1 = \Sigma X_1'$, and $X_2 = \Sigma X_2'$ be suspensions of connected $CW$-complexes. Let $c$ be a power sequence such that $c^{\{i\}} = 1$ for $i = 1,2$, and $c^{\{1,2\}} = (a,b)$ for some positive integers $a$ and $b$. Denote by $D_{ab}$ the homotopy cofibre of the degree $ab$ map on $\Sigma(X_1' \wedge X_2')$. Then there is a homotopy fibration \[D_{ab} \rtimes (\Omega X_1 \times \Omega X_2) \rightarrow \ux^{\Delta^1,c} \xrightarrow{f} X_1 \times X_2,\] which splits after looping.
\end{lemma}
\begin{proof}
Let $i_1: X_1 \to X_1 \vee X_2$ and $i_2:X_2 \to X_1 \vee X_2$ be the inclusions of the first and second wedge summands respectively. By \cite[Example 2.26]{weighted}, there is a homotopy cofibration \[\Sigma(X_1' \wedge X_2') \xrightarrow{ab[i_1,i_2]} X_1 \vee X_2 \xrightarrow{f} X_1 \times X_2.\]

    Consider the diagram \[\begin{tikzcd}
	{\Sigma(X_1' \wedge X_2')} & {\Sigma(X_1' \wedge X_2')} & {D_{ab}} \\
	{\Sigma(X_1' \wedge X_2')} & {X_1 \vee X_2} & {\ux^{\Delta^1,c}} \\
	& {X_1 \times X_2} & {X_1 \times X_2.}
	\arrow["ab", from=1-1, to=1-2]
	\arrow[equals, from=1-1, to=2-1]
	\arrow[from=1-2, to=1-3]
	\arrow["{[i_1,i_2]}", from=1-2, to=2-2]
	\arrow["h", from=1-3, to=2-3]
	\arrow["{ab[i_1,i_2]}", from=2-1, to=2-2]
	\arrow[from=2-2, to=2-3]
	\arrow["f", from=2-2, to=3-2]
	\arrow["g", from=2-3, to=3-3]
	\arrow[equals, from=3-2, to=3-3]
\end{tikzcd}\] The top left square clearly homotopy commutes. Taking homotopy cofibres horizontally, we obtain the top right square. Moreover, the top right square is a homotopy pushout, and so taking homotopy cofibres vertically, we obtain the bottom right square.

 It is well known that $\Omega f$ has a right homotopy inverse, and so homotopy commutativity of the bottom right square implies $\Omega g$ has a right homotopy inverse. Therefore, \cite[Proposition 3.5]{BT} gives the asserted homotopy fibration.
\end{proof}

\begin{prp}
\label{prp:torsion-in-1simplex}
Let $X_1 = \Sigma X_1'$, and $X_2 = \Sigma X_2'$ be suspensions of connected $CW$-complexes. Let $c$ be a power sequence such that $c^{\{i\}} = 1$ for $i = 1,2$, and $c^{\{1,2\}} = (a,b)$ for some positive integers $a$ and $b$. Then $\tau(\Omega\ux^{\Delta^1,c})\subseteq\tau(X'_1)\cup\tau(X'_2)\cup\{\text{prime divisors of }ab\}$.
\end{prp}
\begin{proof}
The degree $ab$ map on $\Sigma(X_1'\wedge X_2')$ is a homotopy equivalence after localisation away from prime divisors of $ab$ and rationally. It follows that $D_{ab}$ is a simply connected $ab$-torsion space, hence we obtain $\tau(\Omega D_{ab})\subseteq\{\text{prime divisors of }ab\}$.

By \cite[Lemma 2.12]{S}, $\Omega(A\rtimes B)\simeq \Omega A\times\Omega\Sigma(\Omega A\wedge B)$, hence $\tau(\Omega(A\rtimes B))= \tau(\Omega A)\cup\tau(B)$ by \cite[Lemma 6.5]{SV}. It follows that $\tau(\Omega(D_{ab}\rtimes (\Omega X_1\times\Omega X_2)))=\tau(\Omega D_{ab})\cup\tau(\Omega X_1)\cup\tau(\Omega X_2)$. Moreover, \cite[Lemma 6.5 (6)]{SV} implies that $\tau(\Omega X_1)\cup\tau(\Omega X_2) = \tau(X'_1) \cup \tau(X'_2)$. The result now follows from Proposition \ref{prp:split_hopf_properties}.
\end{proof}
\begin{crl}
Suppose that $X_i=\Sigma X_i'$ are suspensions of connected CW-complexes, $c^{\{i\}}=1$ for all $i\in[m]$ and $\K$ is a tree. Then $$\tau(\Omega\ux^{\K,c})\subseteq\bigcup_{i=1}^m\tau(X'_i)\cup\bigcup_{\{i,j\}\in\K}\{\text{prime divisors of }m_{ij}\}.$$
\end{crl}
\begin{proof}
By Theorem \ref{thm:torsion-reduction-to-neighbourly}, it is sufficient to prove this for $\K=\Delta^1$. The result then follows from Proposition \ref{prp:torsion-in-1simplex}.
\end{proof}

When each $X_i$ is a sphere, we can refine these results further. For integers $m \geq 1$, $r \geq 1$, $d \geq 3$, recall that $P^d(m)$ is the mod $m$ Moore space defined as the mapping cone of the degree $m$ map on $S^{d-1}$. In this case, $D_{ab} = P^{d_1+d_2}(ab)$. Moreover, if $ab=p_1^{r_1}\cdots p_n^{r_n}$ is a prime decomposition of $ab$, then there is a homotopy equivalence $P^{d_1+d_2}(ab) = \bigvee_{i=1}^n P^{d_1+d_2}(p_i^{r_i})$. Note that $P^{d}(p_i^{r_i})$ is contractible rationally and localised at primes $q \neq p_i$.

\begin{crl}
\label{crl:spherecasemoore}
Let $d_1,d_2\geq 2.$
Let $c$ be a power sequence such that $c^{\{i\}} = 1$ for $i = 1,2$, and $c^{\{1,2\}} = (a,b)$ for some positive integers $a$ and $b$. Then there is a homotopy fibration \[P^{d_1+d_2}(ab) \rtimes (\Omega S^{d_1}\times\Omega S^{d_2}) \rightarrow \us^{\Delta^1,c} \xrightarrow{f} S^{d_1}\times S^{d_2}\] which splits after looping.
In particular,
$$\Omega\us^{\Delta^1,c}\in\prod(\mathcal{P\cup T}),\quad\tau(\Omega\us^{\Delta^1,c})=\{\text{prime divisors of }ab\}.$$
\end{crl}
\begin{proof}
The asserted homotopy fibration follows from Lemma \ref{lem:1simplexdecomp}. From the definition of the Moore space, $\tau(P^{d_1+d_2}(ab)) = \{\text{prime divisors of }ab\}$. The torsion result then follows by arguing as in Proposition \ref{prp:torsion-in-1simplex}.

It remains to show that $\Omega\us^{\Delta^1,c}\in\prod(\mathcal{P\cup T})$. Lemma \ref{lem:1simplexdecomp} implies there is a homotopy equivalence \[\Omega \us^{\Delta^1,c} \simeq \Omega S^{d_1} \times \Omega S^{d_2} \times \Omega(P^{d_1+d_2}(ab) \rtimes (\Omega S^{d_1}\times\Omega S^{d_2})).\] Hence, it suffices to show that $\Omega(P^{d_1+d_2}(ab) \rtimes (\Omega S^{d_1}\times\Omega S^{d_2})) \in \prod(\mathcal{P\cup T})$. However, this follows from \cite[Lemma 5.2 (2) and (5)]{S1}.
\end{proof}

\begin{rmk}
The homotopy fibration and corresponding splitting in Corollary \ref{crl:spherecasemoore} can be obtained as a special case of \cite[Theorem 1.1]{H}. Lemma \ref{lem:1simplexdecomp} can be viewed as a mild generalisation of this result in the context of weighted polyhedral products.
\end{rmk}

We can now state the loop space decomposition for $\Omega \us^{\K,c}$, where $\K$ is a tree. 

\begin{theorem}
\label{thm:tree-decomposition}
Let $\K$ be a tree on $[m]$. Let $c$ be a power sequence associated to $\K$ such that $c^{\{i\}}=1$. Denote $m_{ij}=c^{\{i,j\}}_i\cdot c^{\{i,j\}}_j$ for $\{i,j\}\in\K$. Then \[\Omega \us^{\K,c} \in \prod(\mathcal{P} \cup \mathcal{T}),\quad\tau(\Omega \us^{\K,c}) = \bigcup_{\{i,j\}\in\K}\{\text{prime divisors of }m_{ij}\}.\]
\end{theorem}

\begin{proof}
By Theorem \ref{thm:homotopy-decomp-reduction-to-neighbourly} and Theorem \ref{thm:torsion-reduction-to-neighbourly}, it suffices to show the result for $\Omega \us^{\Delta^1,c}$. However, this follows from Corollary \ref{crl:spherecasemoore}.
\end{proof}

\end{document}